\documentclass[review,1p,12pt]{elsarticle}

\usepackage{amsmath,graphicx,epsfig,amssymb,fullpage,longtable}
\usepackage{pifont,booktabs,multirow}
\newcommand{\argmin}{\operatornamewithlimits{argmin}}
\newcommand{\argmax}{\operatornamewithlimits{argmax}}

\newtheorem{theorem}{Theorem}[section]
\newtheorem{lemma}[theorem]{Lemma}
\newdefinition{definition}{Definition}

\newproof{proof}{Proof}

\begin{document}

\begin{frontmatter}

\title{Recovering Non-negative and Combined Sparse Representations}

\author{Karthikeyan~Natesan~Ramamurthy \corref{corresp}}
\author{Jayaraman~J.~Thiagarajan}
\author{Andreas~Spanias}

\address{SenSIP Center, School of ECEE, Arizona State University, Tempe, AZ 85287-5706 USA.}

\cortext[corresp]{
Corresponding author:
Email: knatesan@asu.edu, Phone: 001-(480)-297-6883, Fax: 001-(480)-965-8325
}

\begin{abstract}
The non-negative solution to an underdetermined linear system can be uniquely recovered sometimes, even without imposing any additional sparsity constraints. In this paper, we derive conditions under which a unique non-negative solution for such a system can exist, based on the theory of polytopes. Furthermore, we develop the paradigm of combined sparse representations, where only a part of the coefficient vector is constrained to be non-negative, and the rest is unconstrained (general). We analyze the recovery of the unique, sparsest solution, for combined representations, under three different cases of coefficient support knowledge: (a) the non-zero supports of non-negative and general coefficients are known, (b) the non-zero support of general coefficients alone is known, and (c) both the non-zero supports are unknown. For case (c), we propose the combined orthogonal matching pursuit algorithm for coefficient recovery and derive the deterministic sparsity threshold under which recovery of the unique, sparsest coefficient vector is possible. We quantify the order complexity of the algorithms, and examine their performance in exact and approximate recovery of coefficients under various conditions of noise. Furthermore, we also obtain their empirical phase transition  characteristics. We show that the basis pursuit algorithm, with partial non-negative constraints, and the proposed greedy algorithm perform better in recovering the unique sparse representation when compared to their unconstrained counterparts. Finally, we demonstrate the utility of the proposed methods in recovering images corrupted by saturation noise.
\end{abstract}


\end{frontmatter}
\journal{Digital Signal Processing}
\section{Introduction}
\label{sec:intro}
We investigate the problem of recovering non-negative and combined sparse representations from underdetermined linear models. The system of linear equations with the constraint that the solution is non-negative can be expressed as 
\begin{equation}
\label{eqn:nonneg_rep}
\mathbf{y} = \mathbf{X} \boldsymbol{\alpha}, \text{ where } \boldsymbol{\alpha} \geq 0,
\end{equation}where $\mathbf{y} \in \mathbb{R}^{M}$  is the data vector, $\boldsymbol{\alpha} \in \mathbb{R}^{K_x}$ is the non-negative solution (coefficient vector) and $\mathbf{X} \in \mathbb{R}^{M \times K_x}$ is the dictionary with $K_x > M$. When only a part of the solution is constrained to be non-negative and the rest is unconstrained (general), we obtain the combined representation model, 
\begin{equation}
\label{eqn:nonneg_rep_sparse_corr}
\mathbf{y} = \mathbf{X} \boldsymbol{\alpha} + \mathbf{D} \boldsymbol{\beta}, \text{ where } \boldsymbol{\alpha} \geq 0.
\end{equation} Here, the coefficient vector $\boldsymbol{\beta} \in \mathbb{R}^{K_d}$ is unconstrained, and $\mathbf{X} \in \mathbb{R}^{M \times K_x}$ and $\mathbf{D} \in \mathbb{R}^{M \times K_d}$ are the sub-dictionaries for the non-negative and general representations respectively. We denote the combined coefficient vector as $\boldsymbol{\delta} = [\boldsymbol{\alpha}^T \text{ } \boldsymbol{\beta}^T]^T$, and the combined dictionary as $\mathbf{G}=[\mathbf{X} \quad \mathbf{D}]$. We assume that $\mathbf{G}$ is overcomplete with $K_x+K_d > M$, and the columns of the dictionaries are normalized to have unit $\ell_2$ norm. The sparsest solutions to (\ref{eqn:nonneg_rep}) and (\ref{eqn:nonneg_rep_sparse_corr}) are obtained by minimizing the $\ell_0$ norm, the number of non-zero elements, of the corresponding non-negative coefficient vector ($\boldsymbol{\alpha}$) or the combined coefficient vector ($\boldsymbol{\delta}$). In both the cases, the unique minimum $\ell_0$ norm solution, when it exists, will be referred to as $ML_0$ solution. In this paper, we focus on obtaining deterministic guarantees for recovery of the $ML_0$ solutions to the linear systems (\ref{eqn:nonneg_rep}) and (\ref{eqn:nonneg_rep_sparse_corr}), using both convex and greedy algorithms, based on the properties of the dictionaries.

\subsection{Applications}
\label{sec:app_models}
Some of the applications of the non-negative representation model in (\ref{eqn:nonneg_rep}), and the combined model in (\ref{eqn:nonneg_rep_sparse_corr}) are in image recovery \cite{Ramamurthy2011}, automatic speech recognition using exemplars \cite{Gemmeke2011}, protein mass spectrometry \cite{Slawski2010a}, astronomical imaging \cite{Bardsley2006}, spectroscopy \cite{Donoho1992}, source separation \cite{Benaroya2003}, clustering/semi-supervised learning of data \cite{Cheng2010,He_Semi}, sparse portfolio optimization \cite{Brodie2008} to name a few. In particular, we will briefly mention two applications where the proposed combined model is directly relevant. 

\subsubsection{Signal/Image Recovery}
\label{sec:im_rec_app}
Natural image patches can be sparsely represented using predefined and learned dictionaries and this property is used favorably in many image recovery applications such as denoising, inpainting, super-resolution and compressed sensing. When the representation of the image has two components, which are sparse in two distinct dictionaries, and when the sign of the coefficients in one of the dictionaries is known, the proposed combined sparse models can be used to recover the coefficients and hence the image itself. One such example application for the proposed model is demonstrated in Section \ref{sec:image_patch_rec}, where we recover images corrupted by saturation noise. Another potential application is in compressed recovery of sparse signals, when the signs of a subset of the coefficient vector is known. The utility of the combined model in this application is illustrated in Sections \ref{sec:exact_rec}, \ref{sec:approx_rec} and \ref{sec:phase_tran}.

\subsubsection{Sparse Markowitz Portfolio Optimization}
\label{sec:sparse_marko}
In portfolio optimization, the goal is to select assets for a given capital that balances high returns with low risk. Recently, it has been proposed that this can be solved as a sparse optimization problem with appropriate constraints \cite{Brodie2008}. In this context, a negative coefficient corresponds to a \textit{short-position} on the portfolio and a non-negative coefficient corresponds to a \textit{no-short position}. When certain positions are mandated to be no-shorts because of possible government or market regulations, the combined model can be effectively used to select an optimal portfolio.

\subsection{Prior Work}
\label{sec:prior_work}
For the non-negative representation model in (\ref{eqn:nonneg_rep}), a sufficiently sparse $ML_0$ solution can be recovered by minimizing the $\ell_1$ norm of $\boldsymbol{\alpha}$, using the non-negative version of the basis pursuit (BP) algorithm \cite{Chen_BP}, which we refer to as NN-BP. The optimization program can be  expressed as
\begin{equation}	
\label{eqn:l1_min}
\min_{\boldsymbol{\alpha}} \mathbf{1}^T \boldsymbol{\alpha} \text{ subject to } \mathbf{y} = \mathbf{X} \boldsymbol{\alpha}, \boldsymbol{\alpha} \geq \mathbf{0}.
\end{equation} The conditions on $\mathbf{X}$ under which the recovery of $ML_0$ solution using (\ref{eqn:l1_min}) is possible have been derived based on the neighborliness of polytopes \cite{DonohoNN,Donoho2010,Wang2010}, and the non-negative null-space property \cite{Khajehnejad2009}. A non-negative version of the greedy orthogonal matching pursuit (OMP) algorithm \cite{Tropp2004}, which we will refer to as NN-OMP, for recovering the coefficients has also been proposed \cite{Bruckstein2008a}. If the set \begin{equation}
\label{eqn:nn_rec_uniq}
\{\boldsymbol{\alpha} | \mathbf{y} = \mathbf{X}\boldsymbol{\alpha}, \boldsymbol{\alpha} \geq 0 \}
\end{equation} contains only one solution, we can use any variational function instead of the $\ell_1$ norm in order to obtain the unique non-negative solution  \cite{Donoho2010, Wang2010,Bruckstein2008a}. In particular, the solution can be obtained by using the non-negative least squares (NNLS) algorithm \cite{Slawski2010a, Slawski2011}.

A major part of our work investigates the combined sparse representation model introduced in (\ref{eqn:nonneg_rep_sparse_corr}), where only a part of the sparse coefficient vector is constrained to be non-negative. We consider the \textit{deterministic sparsity thresholds} i.e., the maximum number of non-zero coefficients possible in the $ML_0$ solution, such that the $ML_0$ solution can be uniquely recovered. To the best of our knowledge, such an investigation has not been reported so far in the literature. However, when both $\boldsymbol{\alpha}$ and $ \boldsymbol{\beta}$ are unconstrained general sparse vectors, the sparsity thresholds for recovery of the $ML_0$ solution have been presented in \cite{Kuppinger2011, Studer2011}. By considering the coherence parameters of $\mathbf{X}$ and $\mathbf{D}$ separately, the authors in \cite{Kuppinger2011} show that an improvement up to a factor of two can be achieved in the deterministic sparsity threshold when compared to considering $\mathbf{X}$ and $\mathbf{D}$ together as a single dictionary. Note that deterministic sparsity thresholds provide guarantees that hold for all sparsity patterns and non-zero values in the coefficient vectors. \textit{Probabilistic} or \textit{robust} sparsity thresholds, that hold for most sparsity patterns and non-zero values in the coefficient vectors have also been derived in \cite{Kuppinger2011}, again for the case where $\boldsymbol{\alpha}$ and $ \boldsymbol{\beta}$ are general sparse vectors. When this representation is approximately sparse and corrupted by additive noise, theory and algorithms for coefficient recovery are presented in \cite{Studer2011a}.

\subsection{Contributions}
\label{sec:contrib}
We present deterministic recovery guarantees for both the non-negative and the combined sparse representation models given by (\ref{eqn:nonneg_rep}) and (\ref{eqn:nonneg_rep_sparse_corr}) respectively. Furthermore, we propose a greedy algorithm for performing coefficient recovery in combined representations and derive deterministic sparsity thresholds for unique recovery using $\ell_1$ minimization and the proposed greedy algorithm.

For the non-negative model in (\ref{eqn:nonneg_rep}), we derive the sufficient conditions for (\ref{eqn:nn_rec_uniq}) to be singleton based on the neighborliness properties of the quotient polytope corresponding to the dictionary $\mathbf{X}$. Similar analyses reported in \cite{Donoho2010,Wang2010} assume that the dictionary $\mathbf{X}$ is obtained from a random ensemble and append a row of ones to it, such that the row span of $\mathbf{X}$ contains the vector $\mathbf{1}^T$. In contrast, we do not assume any randomness on $\mathbf{X}$ and only require that its row span intersects the positive orthant. We show that the sparsity threshold on $\boldsymbol{\alpha}$, for the set (\ref{eqn:nn_rec_uniq}) to be singleton, is the same as the deterministic sparsity threshold for recovering the $ML_0$ solution of a general sparse representation. Whenever this threshold is satisfied, $\ell_1$-norm regularization in (\ref{eqn:l1_min}) can be replaced with any variational function. Section \ref{sec:nonneg_rep} presents the analysis of the non-negative representation model.

For the combined model in (\ref{eqn:nonneg_rep_sparse_corr}), we propose a variant of the greedy OMP algorithm, the combined OMP (COMB-OMP) algorithm, for performing coefficient recovery. We also consider a $\ell_1$ regularized convex algorithm, which we refer to as combined BP (COMB-BP). We derive the deterministic sparsity thresholds for recovering the $ML_0$ solution using both the COMB-BP and COMB-OMP algorithms. We show that a factor-of-two improvement in the sparsity threshold, observed when $\boldsymbol{\alpha}$ and $\boldsymbol{\beta}$ are general sparse vectors \cite{Kuppinger2011}, holds for recovery using the COMB-BP also. We also show that such an improvement in the sparsity threshold cannot be observed using the COMB-OMP algorithm, because of the partial non-negativity constraint on the coefficient vector. Furthermore, we obtain the sparsity thresholds in the following cases of coefficient support knowledge: (a) the non-zero support of both $\boldsymbol{\alpha}$, $\boldsymbol{\beta}$ are known, and (b) non-zero support of $\boldsymbol{\beta}$ alone is known. When analyzing case (b), we factor out the contribution of the general representation component and arrive at conditions under which $\ell_1$-norm regularization in the resulting optimization can be replaced with any variational function for the recovery of $\boldsymbol{\alpha}$. Section \ref{sec:non_neg_sparse_corr} presents all the details in the analysis of the combined representation model. As a final piece of our theoretical investigation, we present the computational complexities of OMP, COMB-OMP, BP and COMB-BP algorithms.

The performance of the COMB-BP and the COMB-OMP algorithms are also analyzed using simulations. The dictionary $\mathbf{G}$ is obtained from a Gaussian ensemble and the non-zero coefficients are obtained either from a uniform distribution or from a random sign distribution ($\pm1$). It is shown that both COMB-BP and COMB-OMP respectively perform better than their unconstrained counterparts, the BP and the OMP, particularly as the $K_x$ becomes larger. The algorithms show a similar behavior when recovering the sparse coefficients from noisy signals. Furthermore, the empirical phase transition characteristics of the proposed algorithms are provided and their utility in a real-world application of recovering images from saturation noise is demonstrated.

\subsection{Notation}
\label{sec:not}
Lowercase boldface letters denote column vectors and uppercase boldface denote matrices, e.g., $\mathbf{a}$ and $\mathbf{A}$ denote a vector and a matrix respectively. $\mathbf{a}_i$ indicates the $i^{th}$ column of the matrix $\mathbf{A}$. The Moore-Penrose pseudoinverse of a matrix $\mathbf{A}$, given by $(\mathbf{A}^T \mathbf{A})^{-1} \mathbf{A}^T$  is denoted as $\mathbf{A}^{\dagger}$. The notation $\mathbf{A} = diag(\mathbf{a})$ means that $\mathbf{A}$ is a diagonal matrix with the elements of the vector $\mathbf{a}$ as its diagonal. $|\mathbf{A}|$ refers to a matrix whose elements are the absolute values of the elements of $\mathbf{A}$ and the same notation applies to vectors also. The maximum row and maximum column sums of $\mathbf{A}$ are referred to as $\|\mathbf{A}\|_{\infty,\infty}$ and $\|\mathbf{A}\|_{1,1}$ respectively. A set is denoted as $\mathcal{A}$, its cardinality is given by $|\mathcal{A}|$ and its complement by $\mathcal{A}^c$. The operator $[.]^{+}$ returns the maximum of the argument and zero, and $\max(.,.)$ returns the maximum of the two arguments. $\mathbf{I}_K$ denotes an identity matrix of size $K \times K$, $\mathbf{1}_{K_1,K_2}$ is a matrix of ones with size $K_1 \times K_2$. Similar notation will be employed for defining vectors also. When it is clear from context, the subscripts will be dropped for simplicity.

\section{Non-negative Sparse Representations}	
\label{sec:nonneg_rep}
For the non-negative representation given in (\ref{eqn:nonneg_rep}), we denote the number of non-zero coefficients in $\boldsymbol{\alpha}$ as $S_x$. 
\begin{definition} (\cite{Donoho2003a})
The two-sided coherence (or simply coherence) of the dictionary $\mathbf{X}$ is
\begin{equation}
\label{eqn:two_sided_coh}
{\mu}_x = \max_{i \neq j} \frac{|\mathbf{x}_i^T\mathbf{x}_j|}{\|\mathbf{x}_i\|_2 \|\mathbf{x}_j\|_2},
\end{equation}
\end{definition}
\begin{definition} (\cite{Bruckstein2008a})
The one-sided coherence of the dictionary $\mathbf{X}$ is
\begin{equation}
\label{eqn:one_sided_coh}
{\sigma}_x = \max_{i \neq j} \frac{|\mathbf{x}_i^T\mathbf{x}_j|}{\|\mathbf{x}_i\|_2^2}.
\end{equation}
\end{definition}
If the columns of $\mathbf{X}$ are normalized, we have $\mu_x = \sigma_x$, and if they had different $\ell_2$ norms, we would have  $\mu_x \leq \sigma_x$ \cite[Lemma 1]{Bruckstein2008a}. 
\begin{definition} (\cite{Bruckstein2008a})
The dictionary $\mathbf{X}$ belongs to the class of matrices denoted as $\mathcal{M}^{+}$, if its row span intersects the positive orthant. 
\end{definition}
If $\mathbf{X} \in \mathcal{M}^{+}$, $\exists \mathbf{h}$ such that  $\mathbf{h}^T \mathbf{X} = \mathbf{w}^T,  \mathbf{w}> \mathbf{0}$. Let us define $\mathbf{W} = diag(\mathbf{w})$, $\mathbf{U} = \mathbf{X}\mathbf{W}^{-1}$, and denote $\sigma_u$ and $\mu_u$ as the one-sided and two-sided coherences of $\mathbf{U}$ respectively. In \cite[Theorem 1]{Wang2010} it is shown that $\mathbf{X} \in \mathcal{M}^+$ is a necessary condition for (\ref{eqn:nn_rec_uniq}) to be singleton. The main result in \cite[Theorem 2]{Bruckstein2008a} states that the set (\ref{eqn:nn_rec_uniq}) will be singleton if $\mathbf{X} \in \mathcal{M}^{+}$ and $S_x < 0.5(1+1/\sigma_u)$.

We will now state the main result of this section, whose proof will be relegated to the end of this section.
\begin{theorem}
\label{thm:non_negrec_coh}
When $\mathbf{X} \in \mathcal{M}^+$, the set defined in (\ref{eqn:nn_rec_uniq}) is singleton if the number of non-zero entries in $\boldsymbol{\alpha}$, $S_x < 0.5(1+1/\mu_x)$. 
\end{theorem}
The threshold given in the above theorem is better than that of \cite[Theorem 2]{Bruckstein2008a}, because $\mu_x = \mu_u$, $\mu_u \leq \sigma_u$ and hence $\mu_x \leq \sigma_u$. We are able to improve the threshold by resorting to geometric arguments based on the theory of polytopes. The rest of this section will state and prove lemmas that will be used in the proof of our main result. 

We will define three geometric entities, the cross-polytope $\mathcal{C}^{K_x}$, the simplex $\mathcal{T}^{K_x-1}$ and the positive orthant $\mathbb{R}^{K_x}_{+}$, that will be used in the proof. The cross-polytope is defined as the $\ell_1$ ball, $\|\boldsymbol{\alpha}\|_1 \leq 1$, in $\mathbb{R}^{K_x}$, and $\mathcal{T}^{K_x-1}$ is the standard simplex, the convex hull of unit basis vectors. Any general sparse representation with $S_x$ non-zero coefficients can be successfully recovered using $\ell_1$ minimization (BP), if the quotient polytope $\mathbf{X} \mathcal{C}^{K_x}$ is \textit{centrally $S_x-$neighborly} \cite[Theorem 1]{Donoho2004}. This form of neighborliness implies that any set of $S_x$ vertices of $\mathbf{X} \mathcal{C}^{K_x}$, not including an antipodal pair (pair of $\pm \mathbf{x}_i$), span a face. For unique recovery of non-negative $S_x-$sparse vectors using the linear program given in (\ref{eqn:l1_min}), the condition on the quotient polytope $\mathbf{X} \mathcal{T}$ is that it must be \textit{outwardly $S_x$-neighborly} \cite[Theorem 1]{DonohoNN}. Here, we fix $\mathcal{T} =  \mathcal{T}^{K_x-1}$ if $\mathbf{0}$ can be expressed as a convex combination of the columns of $\mathbf{X}$, else we fix  $\mathcal{T} =  \mathcal{T}^{K_x}_{\mathbf{0}}$ where $\mathcal{T}^{K_x}_{\mathbf{0}}$ is the solid simplex, the convex hull of $\mathcal{T}^{K_x-1}$ and the origin. When every set of $S_x$ vertices, not including the origin, span a face, the quotient polytope is said to be outwardly $S_x$-neighborly.

\begin{lemma}
\label{lem:non_negrec_neighb}
When $\mathbf{X} \in \mathcal{M}^+$ and the number of non-zero coefficients in $\boldsymbol{\alpha}$ is $S_x$, the set defined in (\ref{eqn:nn_rec_uniq}) is singleton if the quotient polytope $\mathbf{X} \mathcal{T}^{K_x}_{\mathbf{0}}$ is outwardly $S_x-$neighborly. 
\end{lemma}

\begin{proof}
By assumption, $\exists \mathbf{h}$ such that $\mathbf{h}^T\mathbf{X} = \mathbf{w}^T$, $\mathbf{w} > \mathbf{0}$. Consider the quotient polytope $\mathbf{U}\mathcal{T}^{K_x}_{\mathbf{0}}$, where $\mathbf{U} = \mathbf{X}\mathbf{W}^{-1}$ and $\mathbf{W} = diag(\mathbf{w})$. Since $\mathbf{X} \mathcal{T}^{K_x}_{\mathbf{0}}$ is outwardly $S_x-$neighborly and the positive scaling of vertices does not affect the neighborliness of a polytope,  $\mathbf{U}\mathcal{T}^{K_x}_{\mathbf{0}}$ is also outwardly $S_x-$neighborly. Now denote $\hat{\mathbf{y}} = \mathbf{y}/(\mathbf{h}^T \mathbf{y})$ and $\boldsymbol{\gamma} = \mathbf{W}\boldsymbol{\alpha}/(\mathbf{h}^T\hat{\mathbf{y}})$. The set defined in (\ref{eqn:nn_rec_uniq}) has a one-to-one correspondence with 
\begin{equation}
\label{eqn:nn_rec_uniq_mod}
\{\boldsymbol{\gamma} | \hat{\mathbf{y}} = \mathbf{U}\boldsymbol{\gamma}, \boldsymbol{\gamma} \geq 0 \},
\end{equation} If we show that (\ref{eqn:nn_rec_uniq_mod}) is singleton, then (\ref{eqn:nn_rec_uniq}) is singleton as well.

Since we know that $\|\boldsymbol{\alpha}\|_0 = S_x$, this implies that $\|\boldsymbol{\gamma}\|_0 = S_x$. Because of the neighborliness of the quotient polytope $\mathbf{U}\mathcal{T}^{K_x}_{\mathbf{0}}$, $\hat{\mathbf{y}}$ lies in its simplicial face $F$ of affine dimension $S_x$. Denote $\mathcal{V}$ to be the set of vertices of $F$. The remaining vertices in the quotient polytope are denoted by the set $\mathcal{V}^c$. Consider an arbitrary vector $\hat{\mathbf{y}}^c$ expressed as a convex combination of the vertices $\mathcal{V}^c$. Since $F$ is a face, there exists a linear functional $\lambda_F$ and a constant $c$ such that $\lambda_F^T\hat{\mathbf{y}} = c$ and $\lambda_F^T\hat{\mathbf{y}}^c < c$ \cite{Donoho2004}. This means that for an arbitrarily chosen $\hat{\mathbf{y}}$ and $\hat{\mathbf{y}}^c$ which are convex combinations of vertices $\mathcal{V}$ and $\mathcal{V}^c$ respectively, $\|\hat{\mathbf{y}}-\hat{\mathbf{y}^c}\|_2 > 0$. By extension, the rays in the directions of $\hat{\mathbf{y}}$ and $\hat{\mathbf{y}^c}$ intersect only at the origin. The convex cones formed by the vertices $\mathcal{V}$ and $\mathcal{V}^c$ are denoted as $\mathbf{U}_{\mathcal{V}} \mathbb{R}^{|\mathcal{V}|}_{+}$ and $\mathbf{U}_{\mathcal{V}^c} \mathbb{R}^{|\mathcal{V}^c|}_{+}$ respectively. Since  $\|\hat{\mathbf{y}}-\hat{\mathbf{y}^c}\|_2 > 0$ is true for arbitrary pairs of $\hat{\mathbf{y}}$ and $\hat{\mathbf{y}}^c$, the relative interiors of $\mathbf{U}_{\mathcal{V}} \mathbb{R}^{|\mathcal{V}|}_{+}$ and $\mathbf{U}_{\mathcal{V}^c} \mathbb{R}^{|\mathcal{V}^c|}_{+}$ are disjoint. Therefore, from \cite[Theorem 1.32]{Berman2003}, there exists a hyperplane passing through the origin that separates the cones properly. From \cite[Prop. 1]{Slawski2010a}, the existence of such a hyperplane is sufficient for (\ref{eqn:nn_rec_uniq_mod}), and by extension (\ref{eqn:nn_rec_uniq}), to be singleton.
\end{proof}

\subsection{Proof of Theorem \ref{thm:non_negrec_coh}}
\label{sec:pf_thm_non_negrec_coh}
If $S_x < 0.5(1+1/\mu_x)$, the quotient polytope $\mathbf{X} \mathcal{C}^{K_x}$ is centrally $S_x-$neighborly \cite[Corollary 1.1]{Donoho2004}. Since  $\text{vertices}(\mathcal{T}^{K_x}_{\mathbf{0}}) -\{\mathbf{0}\} \subset \text{vertices}(\mathcal{C}^{K_x})$, central $S_x-$neighborliness of $\mathbf{X} \mathcal{C}^{K_x}$ implies outward $S_x-$neighborliness of $\mathbf{X} \mathcal{T}^{K_x}_{\mathbf{0}}$. Note that the vertex $\mathbf{0}$ will be neglected when considering the outward neighborliness. Combining the assumption that $\mathbf{X} \in \mathcal{M}^+$, from Lemma \ref{lem:non_negrec_neighb}, the set defined in (\ref{eqn:nn_rec_uniq}) is singleton.

\section{Combined Sparse Representations}
\label{sec:non_neg_sparse_corr}
We now turn to investigate the problem of combined sparse representations, where a part of the coefficient vector is constrained to be non-negative. For the combined representation model given in (\ref{eqn:nonneg_rep_sparse_corr}), the number of non-zero coefficients and the coefficient support for $\boldsymbol{\alpha}$ are given by $S_x$ and $\mathcal{X}$ respectively. For $\boldsymbol{\beta}$, they are respectively denoted as $S_d$ and $\mathcal{D}$. Let us define the combined representation vector $\boldsymbol{\delta} = [\boldsymbol{\alpha}^T \text{ } \boldsymbol{\beta}^T]^T$ and the combined dictionary $\mathbf{G} = [\mathbf{X} \quad \mathbf{D}]$. The set $\mathcal{G}$ indexes the non-zero coefficients in $\boldsymbol{\delta}$. The length of $\boldsymbol{\delta}$ is denoted by $K_g$ and its number of non-zero coefficients is referred to as $S_g$. We will refer to the coefficient vector  $\boldsymbol{\delta}$ as the combined representation, since it contains both non-negative and general entries from the coefficient vectors $\boldsymbol{\alpha}$ and $\boldsymbol{\beta}$ respectively. We will define the cross-coherence between the matrices $\mathbf{X}$ and $\mathbf{D}$ as
\begin{equation}
\label{eqn:cross_coh}
\mu_g = \max_{i,j} \frac{|\mathbf{x}_i^T \mathbf{d}_j|}{\|\mathbf{x}_i\|_2 \|\mathbf{d}_j\|_2}.
\end{equation} We will present deterministic sparsity thresholds for recovery of the $ML_0$ solution of  (\ref{eqn:nonneg_rep_sparse_corr}) when the coefficient supports are unknown as well as partially known.

\subsection{Non-Zero Supports of $\boldsymbol{\alpha}$ and $\boldsymbol{\beta}$ Known}
\label{sec:both_supp_known}
The vectors $\boldsymbol{\alpha}_1 \in \mathbb{R}^{S_x}$ and $\boldsymbol{\beta}_1 \in \mathbb{R}^{S_d}$ contain the non-zero coefficients of $\boldsymbol{\alpha}$ and $\boldsymbol{\beta}$ indexed by $\mathcal{X}$ and $\mathcal{D}$ respectively. The matrices $\mathbf{X}_1$ and $\mathbf{D}_1$ contain the columns of $\mathbf{X}$ and $\mathbf{D}$ indexed by the sets $\mathcal{X}$ and $\mathcal{D}$ respectively. Since the coefficient supports are known, we can express (\ref{eqn:nonneg_rep_sparse_corr}) as
\begin{equation}
\label{eqn:y_both_supp_known}
\mathbf{y} = \mathbf{X}_1 \boldsymbol{\alpha}_1+\mathbf{D}_1\boldsymbol{\beta}_1,
\end{equation} where $\boldsymbol{\alpha}_1 \geq 0$. We define $\boldsymbol{\delta}_1 = [\boldsymbol{\alpha}_1^T \quad \boldsymbol{\beta}_1^T]^T$ and the matrix $\mathbf{G}_1 = [\mathbf{X}_1 \quad \mathbf{D}_1]$. Recovery can be performed using least squares with inequality constraints (LSI) \cite[Chap. 23]{Lawson1974} as
\begin{equation}
\label{eqn:comb_rec_LSI}
\min_{\boldsymbol{\delta}_1} \|\mathbf{y} - \mathbf{G}_1 \boldsymbol{\delta}_1\|_2 \text{ subject to } \text{ } \mathbf{I}_{\mathcal{X}} \boldsymbol{\delta}_1 \geq \mathbf{0},
\end{equation} where $\mathbf{I}_{\mathcal{X}} = [\mathbf{I}_{S_x} \quad \mathbf{0}_{S_x,S_g}]$ is the indicator matrix such the constraints $\mathbf{I}_{\mathcal{X}} \boldsymbol{\delta}_1 \geq \mathbf{0}$ and $\boldsymbol{\alpha}_1 \geq \mathbf{0}$ are equivalent.

If the matrix $\mathbf{G}_1$ has full column rank, $\boldsymbol{\delta}_1$ can be estimated by just using least squares (LS) instead of LSI, as the additional constraint in (\ref{eqn:comb_rec_LSI}) will not impact the solution. The following theorem presents a sufficient condition for $\mathbf{G}_1$ to be of full column rank.
\begin{theorem} (\cite{Studer2011,Kuppinger2011})
\label{thm:comb_rec_known}
For the system defined in (\ref{eqn:y_both_supp_known}), the matrix $\mathbf{G}_1 = [\mathbf{X}_1 \text{ } \mathbf{D}_1]$ has full column rank if 
\begin{equation}
S_x S_d < \frac{[1-\mu_x(S_x-1)]^{+}[1-\mu_d(S_d-1)]^{+}}{\mu_g^2}.
\end{equation}
\end{theorem}

\subsection{Non-Zero Support of $\boldsymbol{\beta}$ Alone Known}
\label{sec:beta_supp_known}
We will now consider the case where the non-zero support of $\boldsymbol{\beta}$ given by the set $\mathcal{D}$ is known for the system in (\ref{eqn:nonneg_rep_sparse_corr}). We will derive conditions for unique recovery of $\boldsymbol{\alpha}$ using NN-BP and NNLS. With the knowledge of non-zero support of $\boldsymbol{\beta}$, we can rewrite (\ref{eqn:nonneg_rep_sparse_corr}) as
\begin{equation}
\label{eqn:y_beta_supp_known}
\mathbf{y} = \mathbf{X} \boldsymbol{\alpha} + \mathbf{D}_1 \boldsymbol{\beta}_1.
\end{equation} Define $\mathbf{P}_\mathcal{D}$ to be the projection matrix for the subspace orthogonal to the column space of $\mathbf{D}_1$, i.e.,
\begin{equation}
\label{eqn:PD_defn}
 \mathbf{P}_\mathcal{D} = \mathbf{I}_M - \mathbf{D}_1 \mathbf{D}_1^{\dagger}.
\end{equation}
Premultiplying (\ref{eqn:y_beta_supp_known}) with $\mathbf{P}_\mathcal{D}$, we get
\begin{equation}
\label{eqn:nn_rep_sp_err_red}
\mathbf{P}_\mathcal{D} \mathbf{y} = \mathbf{P}_\mathcal{D} \mathbf{X} \boldsymbol{\alpha} \text{   where  } \boldsymbol{\alpha} \geq 0.
\end{equation} Let us define $\tilde{\mathbf{y}} = \mathbf{P}_\mathcal{D} \mathbf{y}$  and $\tilde{\mathbf{X}} = \mathbf{P}_\mathcal{D} \mathbf{X}$, such that (\ref{eqn:nn_rep_sp_err_red}) becomes
\begin{equation}
\label{eqn:nn_rep_sp_err_red1}
\tilde{\mathbf{y}} = \tilde{\mathbf{X}} \boldsymbol{\alpha} \text{   where  } \boldsymbol{\alpha} \geq 0.
\end{equation}
 The condition for recovery of the unique solution $\boldsymbol{\alpha}$ from (\ref{eqn:nn_rep_sp_err_red1}) using NN-BP is
\begin{align}
\label{eqn:cond2_uniq_sol_nn_rep_sp_err}
S_x < 0.5\left( 1+\frac{1}{\mu_{\tilde{x}}}\right),
\end{align}where $\mu_{\tilde{x}}$ is the coherence of $\tilde{\mathbf{X}}$. 

\begin{lemma} (\cite{Studer2011})
\label{lem:coh_red_beta_known}
The coherence of $\tilde{\mathbf{X}}$, given by $\mu_{\tilde{x}}$ can be upper bounded as
\begin{align}
\label{eqn:coh_red_beta_known}
\mu_{\tilde{x}} \leq 0.5 \left(\frac{[1-\mu_d(S_d-1)]^{+}(1+\mu_x)}{\mu_x[1-\mu_d(S_d-1)]^{+}+S_b \mu_g^2}\right).
\end{align}
\end{lemma} The above lemma follows directly from \cite[Theorem 5]{Studer2011}. This also implies that for the existence of $\mathbf{D}_1^{\dagger}$, we need to have $S_d < 1+1/\mu_d$.

\begin{lemma}
\label{lem:necc_cond_beta_known}
Let 
\begin{equation}
\label{eqn:uniq_sol_fulldim}
\{\hat{\boldsymbol{\alpha}} |\mathbf{y} = \mathbf{X}\hat{\boldsymbol{\alpha}}, \hat{\boldsymbol{\alpha}} \geq 0 \} = \{\boldsymbol{\alpha}\}.
\end{equation}
For a given non-zero support set $\mathcal{D}$ of $\boldsymbol{\beta}$
\begin{equation}
\label{eqn:uniq_sol_reddim}
\{\hat{\boldsymbol{\alpha}} | \tilde{\mathbf{y}} = \tilde{\mathbf{X}}\hat{\boldsymbol{\alpha}}, \hat{\boldsymbol{\alpha}} \geq 0 \} = \{\boldsymbol{\alpha}\}
\end{equation} holds if (\ref{eqn:cond2_uniq_sol_nn_rep_sp_err}) is satisfied, and 
\begin{equation}
\label{eqn:cond_red_nnls}
 \exists \mathbf{h} \text{ such that } \mathbf{h}^T \mathbf{X} >0 \text{  and  } \mathbf{h}^T \mathbf{D}_1 = \mathbf{0}.
\end{equation}
\end{lemma}
\begin{proof}
From Theorem \ref{thm:non_negrec_coh}, we know that the singleton condition (\ref{eqn:uniq_sol_reddim}) holds true if (a) the condition in (\ref{eqn:cond2_uniq_sol_nn_rep_sp_err}) is satisfied, and (b) $\exists \mathbf{r}$ such that $\mathbf{r}^T \tilde{\mathbf{X}} > 0$. Since (\ref{eqn:uniq_sol_fulldim}) is true by assumption, $\exists \mathbf{h}$ such that $\mathbf{h}^T \mathbf{X} > 0$. For $\mathbf{h}^T \mathbf{X} > 0$ and $\mathbf{r}^T \tilde{\mathbf{X}} > 0$ to hold together, we should have $\mathbf{h} = \mathbf{P}_\mathcal{D}^T \mathbf{r}$. Therefore, we have $\mathbf{h}^T \mathbf{D}_1 = \mathbf{0}$, following the definition of $\mathbf{P}_{\mathcal{D}}$ in (\ref{eqn:PD_defn}).
\end{proof} If the sufficient conditions in Lemma \ref{lem:necc_cond_beta_known}  are satisfied, NNLS can be used to recover the unique solution of (\ref{eqn:nn_rep_sp_err_red}), for a given non-zero support $\mathcal{D}$ of $\boldsymbol{\beta}$.

\subsection{Non-zero Supports of $\boldsymbol{\alpha}$ and $\boldsymbol{\beta}$ are Unknown}
\label{sec:both_supp_unknown}
When the supports of $\boldsymbol{\alpha}$ and $\boldsymbol{\beta}$ in (\ref{eqn:nonneg_rep_sparse_corr}) are unknown, we will first consider the problem of recovering the coefficients using the convex program,
\begin{equation}
\label{eqn:comb_rec_l1}
\min_{\boldsymbol{\delta}} \|\boldsymbol{\delta}\|_1 \text{ subject to } \mathbf{y} = \mathbf{G} \boldsymbol{\delta}, \text{ }\mathbf{I}_{\bar{\mathcal{X}}} \boldsymbol{\delta} \geq \mathbf{0},
\end{equation} which we refer to as COMB-BP. Here, $\mathbf{I}_{\bar{\mathcal{X}}} = [\mathbf{I}_{K_{x}} \quad \mathbf{0}_{K_x,K_g}]$ is an indicator matrix that picks out $\boldsymbol{\alpha}$ from the vector $\boldsymbol{\delta}$ such that the constraint in (\ref{eqn:comb_rec_l1}) is equivalent to $\boldsymbol{\alpha} \geq 0$. When deriving the threshold on $S_g$ for the recovery of the $ML_0$ solution, without loss of generality, we assume that $S_x \leq S_d$ and $\mu_x \leq \mu_d$. Similar thresholds can be derived for the other cases also.

\subsubsection{Condition for Recovering the $ML_0$ Solution using the Convex Program}
\label{sec:comb_rec_l1_cond}
The sufficient condition for the COMB-BP to recover the $ML_0$ solution is
\begin{equation}
\label{eqn:comb_rec_bp_cond}
\max_{i \in \mathcal{G}^c} \|\mathbf{G}_1^{\dagger} \mathbf{g}_i\|_1 < 1.
\end{equation} This condition is same as the one given in \cite[Theorem 3.3]{Tropp2004a} for recovery of a general sparse vector using BP, since the $\ell_1$ norm does not depend on the sign of the coefficients. From \cite[Theorem 3]{Kuppinger2011}, the condition (\ref{eqn:comb_rec_bp_cond}) can be expressed as
\begin{equation}
\label{eqn:comb_rec_bp_coh}
(1+\mu_d)(2 S_x \mu_d +S_d(\mu_g+\mu_d)) + 2 S_x S_d (\mu_g^2 - \mu_d^2) < (1+\mu_d)^2.
\end{equation} The threshold on the total number of non-zero coefficients, $S_g$, is derived using (\ref{eqn:comb_rec_bp_coh}), and can be found in \cite[Corollary 4]{Kuppinger2011}.
\subsubsection{Condition for Recovering the $ML_0$ Solution using a Greedy Algorithm}
\label{sec:comb_rec_neighb_cond}
We propose a greedy pursuit algorithm that can be used to recover the $ML_0$ solution from (\ref{eqn:nonneg_rep_sparse_corr}). The proposed COMB-OMP algorithm follows a procedure similar to the OMP algorithm \cite{Tropp2004a} and is presented in Table \ref{tab:comb_omp}.  The stopping criterion for this algorithm is either the maximum number of iterations/non-zero coefficients, $T$, or the $\ell_2$ norm of the residual, $\epsilon$. In the algorithm, $\pi(i)$ denotes the correlations computed for the current residual with the normalized atom $\mathbf{g}_i$. When updating the index set of chosen dictionary atoms, $\mathcal{G}_t$, we consider only the positive maximum correlation for atoms corresponding to $\mathbf{X}$ and absolute maximum correlation for atoms corresponding to $\mathbf{D}$. This is consistent with our combined representation scheme. The solution update can be performed using a constrained least squares procedure. The final debiasing step computes the solution using the LSI algorithm described in Section \ref{sec:both_supp_known}. This step will be ignored when deriving the sparsity threshold, since it improves the solution only when the sparsity threshold is not satisfied and  when there is additive noise in the combined model (\ref{eqn:nonneg_rep_sparse_corr}).

\begin{center}
\begin{longtable}{|l|}
\caption{The COMB-OMP Algorithm for Greedy Pursuit of a Combined Representation.}
\label{tab:comb_omp}


\endfirsthead

\hline
\endhead

\hline
\endfoot

\endlastfoot

\hline
\textbf{Goal} \\
Recover the $ML_0$ solution from $\mathbf{y} = \mathbf{G} \boldsymbol{\delta}$ such that $\mathbf{I}_{\bar{\mathcal{X}}} \boldsymbol{\delta} \geq \mathbf{0}$. \\

\textbf{Input} \\
$\mathbf{y}$, the input vector. \\
$\mathbf{G} = [\mathbf{X} \quad \mathbf{D}]$, the combined dictionary. \\
$T$, the desired number of iterations.\\
$\epsilon$, error tolerance. \\

\textbf{Initialization}\\
- Iteration count, $t = 0$.\\
- Solution, $\boldsymbol{\delta}_{t} = 0$.\\
- Residual, $\mathbf{r}_t = \mathbf{y}-\mathbf{G}\boldsymbol{\delta}_t = \mathbf{y}$.\\
- Active coefficient supports, $\mathcal{X}_t = \{\}$, $\mathcal{D}_t = \{\}$, $\mathcal{G}_t = \{\}$.\\
- All coefficient supports, $\bar{\mathcal{X}} = \{i\}_{i=1}^{K_x}$, $\bar{\mathcal{D}} = \{i\}_{i=K_x+1}^{K_g}$, $\bar{\mathcal{G}} = \bar{\mathcal{X}} \cup \bar{\mathcal{D}}$.\\
- Non-active coefficient supports, $\mathcal{X}_t^c = \bar{\mathcal{X}}$, $\mathcal{D}_t^c = \bar{\mathcal{D}}$, $\mathcal{G}_t^c = \bar{\mathcal{G}}$.\\

\textbf{Algorithm}\\
{Loop while } $t \leq T$ \text{ OR } $\|\mathbf{r}_t\|_2 > \epsilon$\\
\quad - \textbf{Compute correlations:} \\
\quad \quad $\pi(i) = \frac{\mathbf{r}_t^T \mathbf{g}_i}{\|\mathbf{g}_i\|_2}$ for $1 \leq i \leq K_g$. \\
\quad - \textbf{Update support:}\\
\quad \quad $\hat{i} = \displaystyle \argmax_{i \in \mathcal{X}_t^c} [\pi(i)]^{+}$. \\
\quad \quad $\hat{j} = \displaystyle \argmax_{j \in \mathcal{D}_t^c}|\pi(j)|$. \\
\quad \quad $\hat{k} = \displaystyle \argmax ([\pi(\hat{i})]^{+},|\pi(\hat{j})|)$. \\
\quad \quad If $\hat{k} \in \mathcal{X}_t^c$, \text{ then } $\mathcal{X}_{t+1} = \mathcal{X}_t \cup \{\hat{k}\}$, \text{ else } $\mathcal{D}_{t+1} = \mathcal{D}_t \cup \{\hat{k}\}$.\\
\quad \quad $\mathcal{G}_{t+1} = \mathcal{X}_{t+1} \cup \mathcal{D}_{t+1}$.\\
\quad - \textbf{Update solution:}\\
\quad \quad $\boldsymbol{\delta}_{t+1} = \argmin_{\boldsymbol{\delta}} \|\mathbf{y}-\mathbf{G}\boldsymbol{\delta}\|_2$ \text{ subject to }  $\textit{\text{support}}(\boldsymbol{\delta}) = \mathcal{G}_{t+1}$, $\text{  } \mathbf{I}_{\mathcal{X}_t} \boldsymbol{\delta} \geq \mathbf{0}$. \\
\quad - \textbf{Update residual: } $\mathbf{r}_{t+1} = \mathbf{y}-\mathbf{G}\boldsymbol{\delta}_{t+1}$.\\
\quad - \textbf{Update support sets: } \\
\quad \quad $\mathcal{G}_{t+1}^c = \bar{\mathcal{G}} - \mathcal{G}_{t+1}, 
\mathcal{X}_{t+1}^c = \bar{\mathcal{X}} - \mathcal{X}_{t+1}, \mathcal{D}_{t+1}^c = \bar{\mathcal{D}} - \mathcal{D}_{t+1}.$\\
\quad - \textbf{Update iteration count: } $t = t+1$.\\
end\\
Debias to compute final $\boldsymbol{\delta}$:\\
$\boldsymbol{\delta}_{t} = \argmin_{\boldsymbol{\delta}} \|\mathbf{y}-\mathbf{G}\boldsymbol{\delta}\|_2$ \text{ subject to }  $\textit{\text{support}}(\boldsymbol{\delta}) = \mathcal{G}_{t}$, $\text{  } \mathbf{I}_{\bar{\mathcal{X}}} \boldsymbol{\delta} \geq \mathbf{0}$.\\
\hline
\end{longtable}
\end{center}

The sufficient sparsity threshold on the coefficient vector under which the COMB-OMP will recover the $ML_0$ solution will be investigated. Some of the strategies used in the proofs are inspired by similar techniques used in \cite{Kuppinger2011,Studer2011,Tropp2004a}. In order to derive the threshold, we will divide the dictionary $\mathbf{G} = [\mathbf{X} \quad \mathbf{D}]$ into four sub-dictionaries $\mathbf{X}_1 \in \mathbb{R}^{M \times S_x}$, $\mathbf{X}_2 \in \mathbb{R}^{M \times (K_x-S_x)}$, $\mathbf{D}_1 \in \mathbb{R}^{M \times S_d}$ and $\mathbf{D}_2 \in \mathbb{R}^{M \times (K_d-S_d)}$. We assume that the matrix $\mathbf{G}_1 = [\mathbf{X}_1 \text{ } \mathbf{D}_1]$ contains the atoms that participate in the representation and $\mathbf{G}_2 = [\mathbf{X}_2 \text{ } \mathbf{D}_2]$ contains those that do not participate. This implies that the signal $\mathbf{y}$ can be represented as
\begin{equation}
\label{eqn:comb_rep_omp}
\mathbf{y} = \mathbf{X}_1 \boldsymbol{\alpha}_1+\mathbf{D}_1 \boldsymbol{\beta}_1,
\end{equation} where the elements of $\boldsymbol{\alpha}_1 \in \mathbb{R}^{S_x}$ are strictly positive and those of $\boldsymbol{\beta}_1 \in \mathbb{R}^{S_d}$ are non-zero. 

\begin{lemma}
\label{lem:comb_omp_l1_cond}
When the matrix $\mathbf{G}_1 = [\mathbf{X}_1 \text{ } \mathbf{D}_1]$ has full column rank, $\mathbf{y}$ is given by (\ref{eqn:comb_rep_omp}), and the residual $\mathbf{r}_t$  of COMB-OMP satisfies
\begin{equation}
\label{eqn:comb_omp_res_rel}
\max(\max(\mathbf{X}_1^T\mathbf{r}_t, \mathbf{0}), \|\mathbf{D}_1^T \mathbf{r}_t\|_\infty) = \|\mathbf{G}_1^T \mathbf{r}_t\|_\infty,
\end{equation} the sufficient condition for COMB-OMP to uniquely recover the $ML_0$ solution from (\ref{eqn:nonneg_rep_sparse_corr}) is
\begin{equation}
\label{eqn:comb_omp_l1_cond}
\max_{i \in \mathcal{G}^c} \|\mathbf{G}_1^{\dagger} \mathbf{g}_i \|_{1} < 1.
\end{equation}
\end{lemma}

\begin{proof}
See \ref{app:comb_omp_l1_cond}.
\end{proof}

Note that the  sufficient condition (\ref{eqn:comb_omp_l1_cond}) given in Lemma \ref{lem:comb_omp_l1_cond} is the same for a general representation also \cite{Kuppinger2011,Tropp2004a}. However, the important difference in the case of recovery using COMB-OMP is that (\ref{eqn:comb_omp_l1_cond}) becomes sufficient only when (\ref{eqn:comb_omp_res_rel}) holds. As we will see in the following lemmas, this will lead to a significant difference in terms of sparsity threshold when compared to COMB-BP. We will first derive conditions under which the first step of the COMB-OMP, when $\mathbf{y} = \mathbf{r}_0$, will satisfy (\ref{eqn:comb_omp_res_rel}). This will then be extended to the residuals at all steps, $\mathbf{r}_t$, where $t \geq 1$.

\begin{lemma}
\label{lem:coh_cond1_comb_omp}
When the matrix $\mathbf{G}_1 = [\mathbf{X}_1 \text{ } \mathbf{D}_1]$ has full column rank, and $\mathbf{y}$ is given as (\ref{eqn:comb_rep_omp}), (\ref{eqn:comb_omp_res_rel}) will be satisfied for $\mathbf{y} = \mathbf{r}_0$ if
\begin{equation}
\label{eqn:coh_cond1_comb_omp}
(S_x - 1)\mu_d + S_d \mu_g < \frac{1}{2}.
\end{equation}
\end{lemma}
\begin{proof}
See \ref{app:coh_cond1_comb_omp}.
\end{proof}

The condition given by (\ref{eqn:coh_cond1_comb_omp}) needs to be satisfied even if there is one non-negative component in the combined representation. For now, let us assume that (\ref{eqn:comb_omp_res_rel}) holds for all $\mathbf{r}_t$, where $t \geq 1$, and derive the threshold on $S_g$ such that the $ML_0$ solution can be recovered from (\ref{eqn:nonneg_rep_sparse_corr}). It will be shown later in the section that the threshold on $S_g$ obtained indeed implies that (\ref{eqn:comb_omp_res_rel}) holds for all $\mathbf{r}_t, t \geq 1$.

\begin{lemma}
\label{lem:rec_cond_comb_omp}
When the matrix $\mathbf{G}_1$ has full column rank, and $\mathbf{y}$ is given as (\ref{eqn:comb_rep_omp}), the sufficient condition for (\ref{eqn:comb_omp_l1_cond}) to be satisfied is
\begin{align}
\label{eqn:rec_cond_comb_omp}
\frac{S_x \mu_d + S_d \mu_g}{1-(S_x \mu_d + S_d \mu_g - \mu_d)} < 1
\end{align}
\end{lemma} 
\begin{proof}
See \ref{app:rec_cond_comb_omp}.
\end{proof}

\begin{figure*}[c]
  \includegraphics[width=11cm]{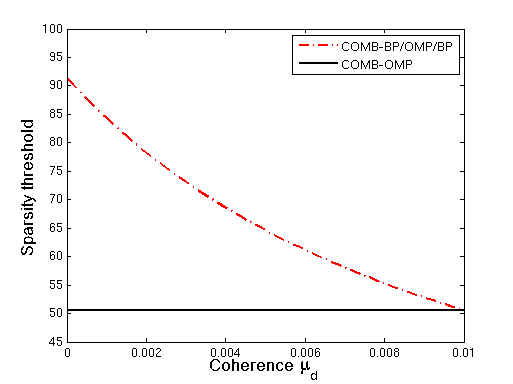}
  \centering
\caption{Deterministic sparsity thresholds for COMB-BP, BP, COMB-OMP and OMP with $\mu_g = 0.01$ in order to recover the $ML_0$ solution from (\ref{eqn:nonneg_rep_sparse_corr}).}
\label{Fig:det_sparse_thresh_comp}
\end{figure*}

When (\ref{eqn:rec_cond_comb_omp}) is satisfied, (\ref{eqn:coh_cond1_comb_omp}) holds as well. Since the number of non-zero coefficients, $S_x$ and $S_d$, of $\boldsymbol{\alpha}_1$ and $\boldsymbol{\beta}_1$ are unknown, we need to derive the condition on recovery that depends only on the number of non-zero coefficients of the combined representation $S_g$.

\begin{lemma}
\label{lem:coh_cond_comb_omp_overall}
When $\mathbf{y}$ is given by (\ref{eqn:comb_rep_omp}), the sparsity threshold on the combined coefficient vector $\boldsymbol{\delta}$ for (\ref{eqn:rec_cond_comb_omp}) to hold is
\begin{equation}
\label{eqn:coh_cond2_comb_omp}
S_g < 0.5\left(1+\frac{1}{\mu_m}\right),
\end{equation} where $\mu_m \equiv \max(\mu_g,\mu_d,\mu_x)$.
\end{lemma}

\begin{proof}
See \ref{app:coh_cond_comb_omp_overall}.
\end{proof} A similar procedure to obtain the threshold on $S_g$ using (\ref{eqn:coh_cond1_comb_omp}), results in $S_g < 0.5 (2+1/\mu_m)$, which is only slightly better than (\ref{eqn:coh_cond2_comb_omp}). Moreover, as stated already (\ref{eqn:rec_cond_comb_omp}) implies (\ref{eqn:coh_cond1_comb_omp}) and not vice-versa. Therefore the sparsity threshold in (\ref{eqn:coh_cond2_comb_omp}) cannot be made better. We will show that the above bound is sufficient for recovering the subsequent atoms using the residuals $\mathbf{r}_t$, when $t \geq 1$.

\begin{lemma}
\label{lem:coh_cond2_comb_omp}
When $\mathbf{y}$ is given by (\ref{eqn:comb_rep_omp}), (\ref{eqn:comb_omp_res_rel}) will hold true for residual at any step, $\mathbf{r}_t$ for $t \geq 1$, when (\ref{eqn:coh_cond2_comb_omp}) is satisfied.
\end{lemma} 
\begin{proof}
See \ref{app:coh_cond2_comb_omp}.
\end{proof}

Now, we are ready to state our main theorem without proof, since it follows directly from the lemmas stated in this section.
\begin{theorem}
\label{thm:comb_model_sparse_thresh}
For any $\mathbf{y}$ that follows the combined model in (\ref{eqn:nonneg_rep_sparse_corr}), COMB-OMP will recover the $ML_0$ solution if the number of non-zero coefficients, $S_g$, is less than $0.5(1+1/\mu_m)$. 
\end{theorem} 

From the lemmas proved in this section, it is clear that this threshold cannot be made better. Contrast this with the case of recovery using a convex program discussed in Section \ref{sec:comb_rec_l1_cond}, as well as sparsity thresholds for recovery using BP and OMP when $\boldsymbol{\alpha}$ and $\boldsymbol{\beta}$ are general sparse vectors \cite[Eqn. (13)]{Kuppinger2011}. Figure \ref{Fig:det_sparse_thresh_comp} compares the thresholds on $S_g$ when $\mu_g = 0.01$ and $\mu_d$ varies from $0$ to $0.01$. We can see that an improvement up to a factor of two can be observed with COMB-BP, BP and OMP algorithms when compared to COMB-OMP. From the proof of Lemma \ref{lem:coh_cond1_comb_omp}, it can be observed that introducing non-negative constraint on even one coefficient in the representation drastically alters the deterministic sparsity threshold of greedy OMP-like algorithms. Note that, however, the sparsity thresholds are pessimistic and in the experiments provided in the Section \ref{sec:exp}, it is shown that COMB-OMP performs better than OMP.

\begin{figure*}[c]
\begin{minipage}[b]{0.5\linewidth}
  \centering
  \includegraphics[width=8cm]{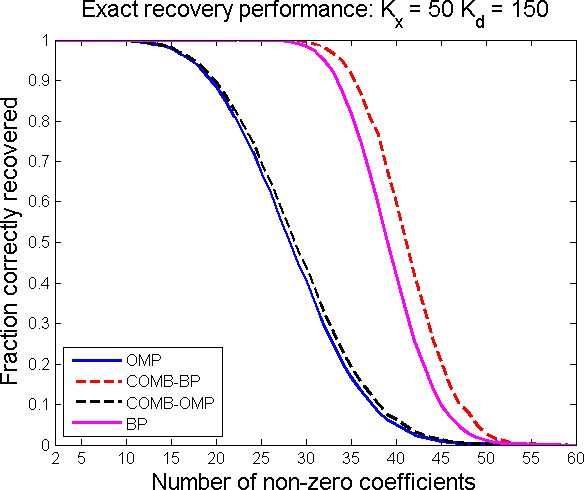}
  \centerline{(a)}\medskip
\end{minipage}
\hfill
\begin{minipage}[b]{0.5\linewidth}
  \centering
  \includegraphics[width=8cm]{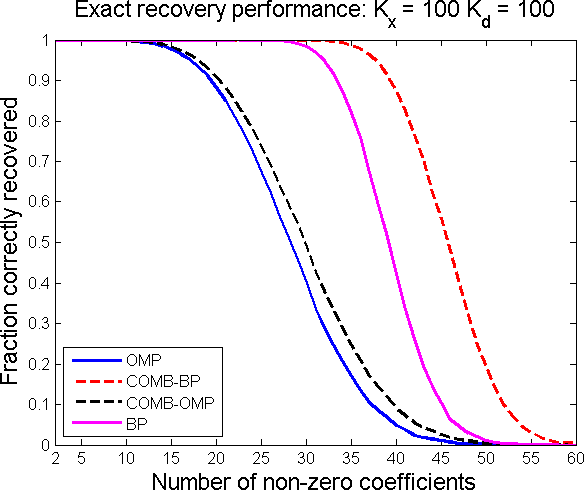}
  \centerline{(b)}\medskip
\end{minipage}
\hfill
\begin{minipage}[b]{1\linewidth}
  \centering
  \includegraphics[width=8cm]{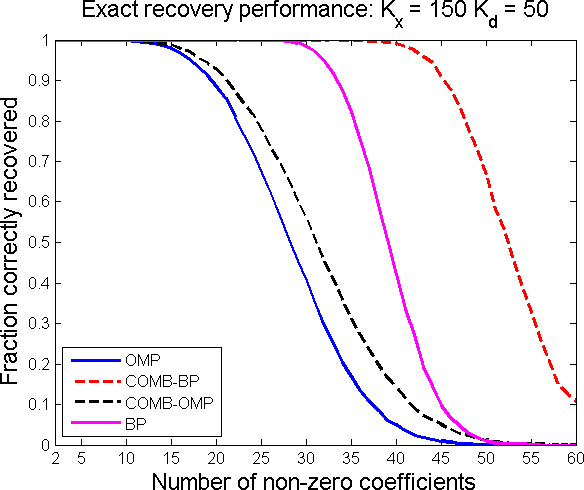}
  \centerline{(c)}\medskip
\end{minipage}
\caption{Exact recovery performance of COMB-BP, BP, COMB-OMP and OMP when $\mathbf{G} \in \mathbb{R}^{M \times K_g}$ $(M=100)$ is obtained from a Gaussian ensemble and the non-zero coefficients are realized from a uniform distribution. (a) $K_x = 50$ and $K_d = 150$, (b) $K_x = 100$ and $K_d = 100$, (c) $K_x = 150$ and $K_d = 50$. In each of the cases and for all the algorithms, the average theoretical sparsity threshold for exact recovery is $1$.}
\label{Fig:gauss_unif_exact_rec}
\end{figure*}

\begin{figure*}[c]
\begin{minipage}[b]{0.5\linewidth}
  \centering
  \includegraphics[width=8cm]{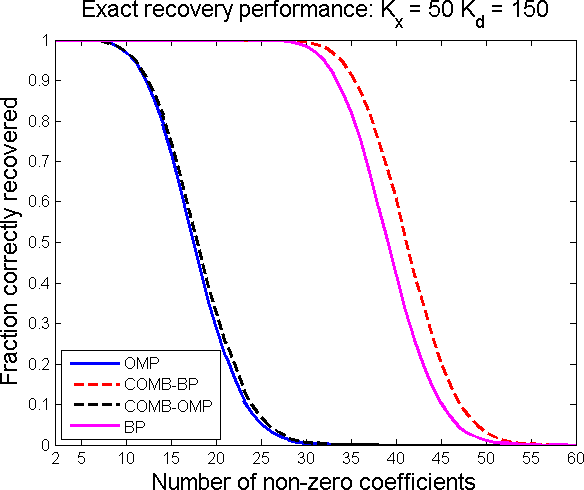}
  \centerline{(a)}\medskip
\end{minipage}
\hfill
\begin{minipage}[b]{0.5\linewidth}
  \centering
  \includegraphics[width=8cm]{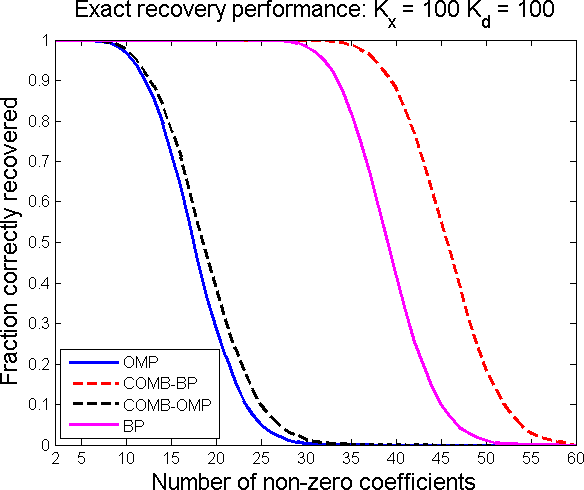}
  \centerline{(b)}\medskip
\end{minipage}
\hfill
\begin{minipage}[b]{1\linewidth}
  \centering
  \includegraphics[width=8cm]{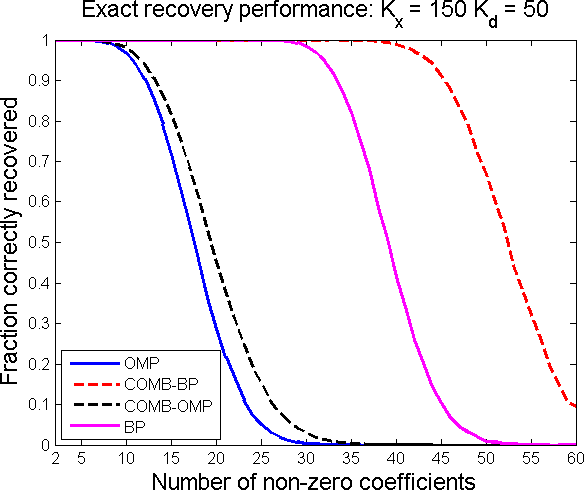}
  \centerline{(c)}\medskip
\end{minipage}
\caption{Exact recovery performance of COMB-BP, BP, COMB-OMP and OMP when $\mathbf{G} \in \mathbb{R}^{M \times K_g}$ $(M=100)$ is obtained from a Gaussian ensemble and the non-zero coefficients are random signs. (a) $K_x = 50$ and $K_d = 150$, (b) $K_x = 100$ and $K_d = 100$, (a) $K_x = 150$ and $K_d = 50$. In each of the cases and for all the algorithms, the average theoretical sparsity threshold for exact recovery is $1$.}
\label{Fig:gauss_signs_exact_rec}
\end{figure*}

\begin{figure*}[c]
\begin{minipage}[b]{0.5\linewidth}
  \centering
  \includegraphics[width=8cm]{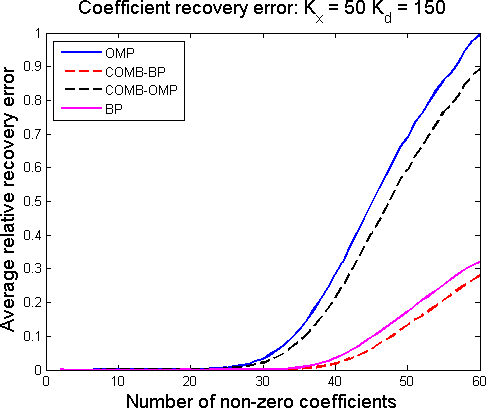}
  \centerline{(a)}\medskip
\end{minipage}
\hfill
\begin{minipage}[b]{0.5\linewidth}
  \centering
  \includegraphics[width=8cm]{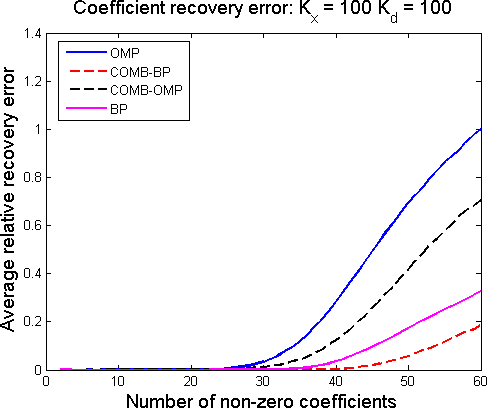}
  \centerline{(b)}\medskip
\end{minipage}
\hfill
\begin{minipage}[b]{1\linewidth}
  \centering
  \includegraphics[width=8cm]{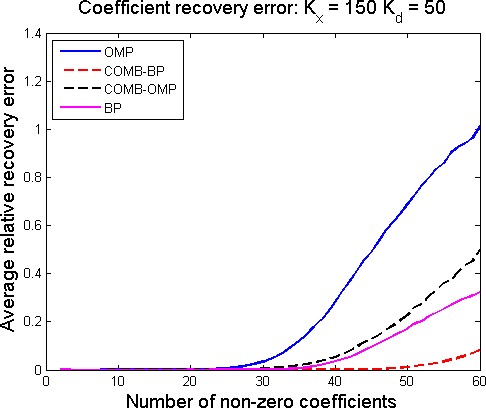}
  \centerline{(c)}\medskip
\end{minipage}
\caption{Average relative recovery error of COMB-BP, BP, COMB-OMP and OMP when $\mathbf{G} \in \mathbb{R}^{M \times K_g}$ $(M=100)$ is obtained from a Gaussian ensemble and the non-zero coefficients are realized from a uniform distribution. (a) $K_x = 50$ and $K_d = 150$, (b) $K_x = 100$ and $K_d = 100$, (a) $K_x = 150$ and $K_d = 50$.}
\label{Fig:gauss_unif_rec_err}
\end{figure*}

\begin{figure*}[c]
\begin{minipage}[b]{0.5\linewidth}
  \centering
  \includegraphics[width=8cm]{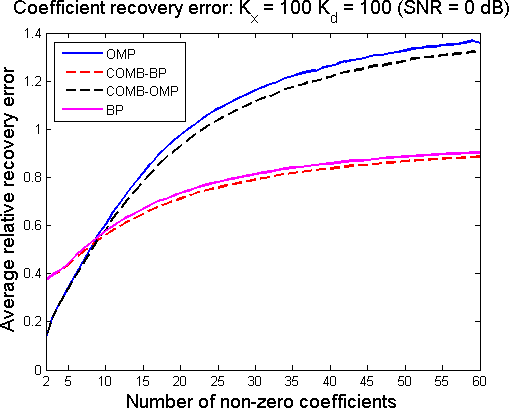}
  \centerline{(a)}\medskip
\end{minipage}
\hfill
\begin{minipage}[b]{0.5\linewidth}
  \centering
  \includegraphics[width=8cm]{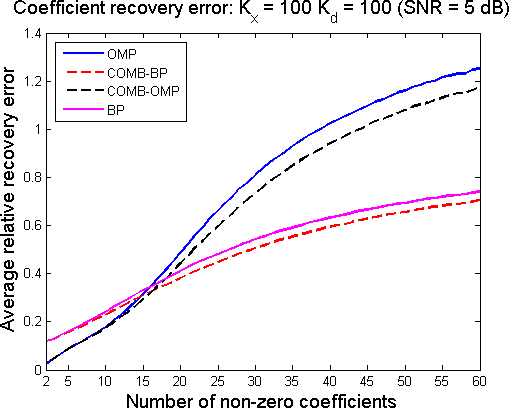}
  \centerline{(b)}\medskip
\end{minipage}
\hfill
\begin{minipage}[b]{0.5\linewidth}
  \centering
  \includegraphics[width=8cm]{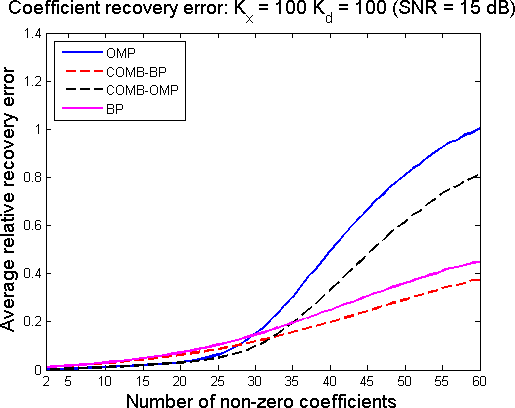}
  \centerline{(c)}\medskip
\end{minipage}
\hfill
\begin{minipage}[b]{0.5\linewidth}
  \centering
  \includegraphics[width=8cm]{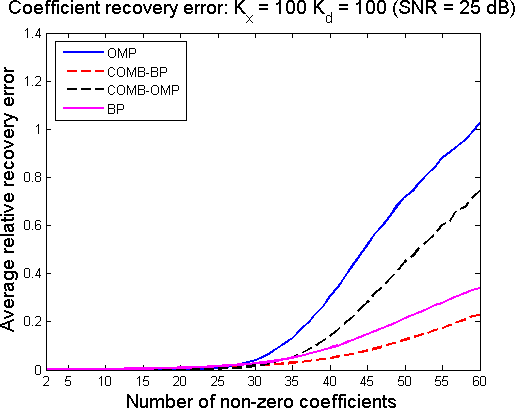}
  \centerline{(d)}\medskip
\end{minipage}
\caption{Average relative recovery error of COMB-BP, BP, COMB-OMP and OMP when $\mathbf{G} \in \mathbb{R}^{M \times K_g}$ $(M=100, K_x = 100, K_d = 100)$ is obtained from a Gaussian ensemble and the non-zero coefficients are realized from a uniform distribution. (a) SNR $= 0$ dB, (b) SNR $= 5$ dB, (c) SNR $= 15$ dB, and (d) SNR $= 25$ dB. }
\label{Fig:gauss_unif_rec_err_noisy}
\end{figure*}

\begin{figure*}[c]
\begin{minipage}[b]{0.5\linewidth}
  \centering
  \includegraphics[width=8cm]{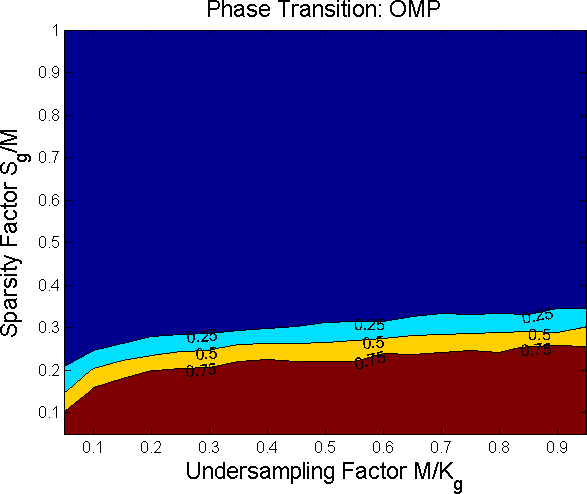}
  \centerline{(a)}\medskip
\end{minipage}
\hfill
\begin{minipage}[b]{0.5\linewidth}
  \centering
  \includegraphics[width=8cm]{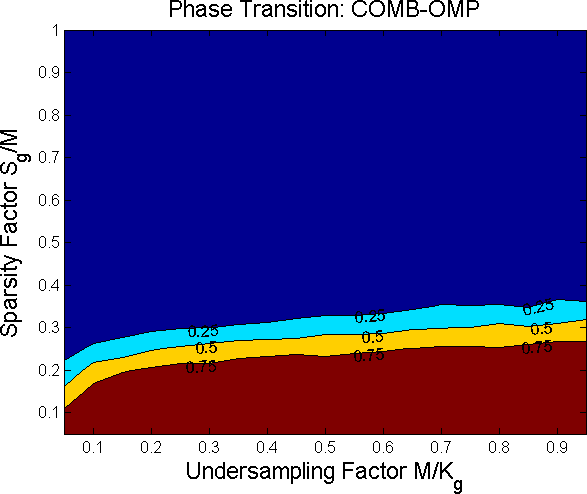}
  \centerline{(b)}\medskip
\end{minipage}
\hfill
\begin{minipage}[b]{0.5\linewidth}
  \centering
  \includegraphics[width=8cm]{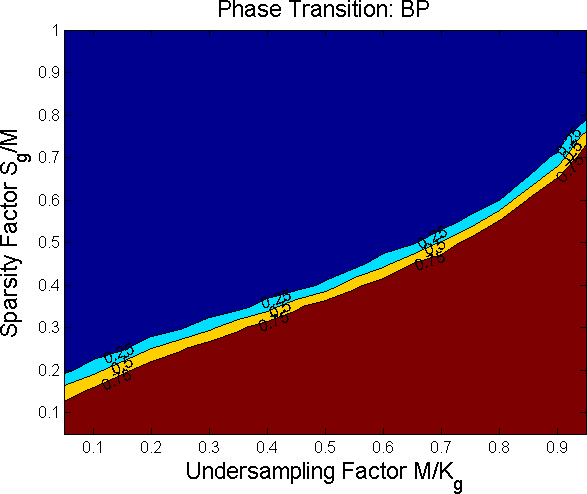}
  \centerline{(c)}\medskip
\end{minipage}
\hfill
\begin{minipage}[b]{0.5\linewidth}
  \centering
  \includegraphics[width=8cm]{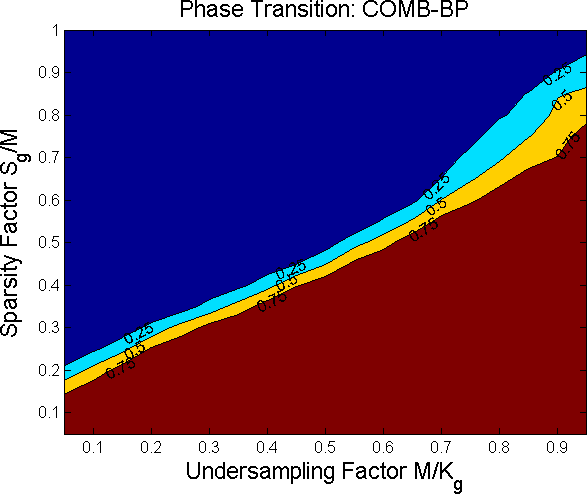}
  \centerline{(d)}\medskip
\end{minipage}
\caption{Phase transition characteristics of OMP, COMB-OMP, BP and COMB-BP when $\mathbf{G} \in \mathbb{R}^{M \times K_g}$ $(K_g = 400, K_x = 200, K_d = 200)$ is obtained from a Gaussian ensemble and the non-zero coefficients are realized from a uniform distribution. The probability of recovery for each contour are also provided in the figures.}
\label{Fig:phase_tr_char}
\end{figure*}

\begin{figure*}[c]
\begin{minipage}[b]{0.32\linewidth}
  \centering
  \includegraphics[width=5cm]{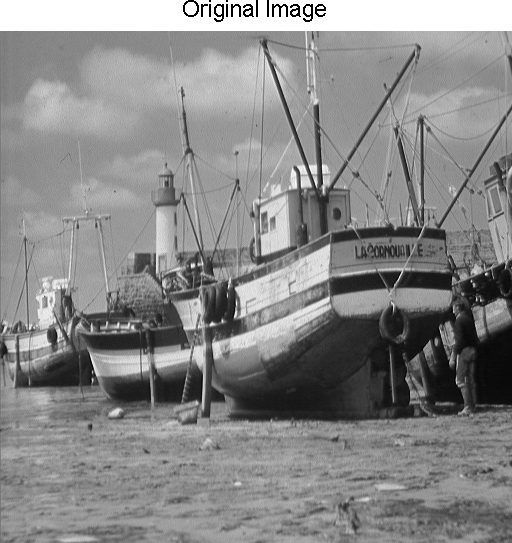}
  \centerline{(a)}\medskip
\end{minipage}
\hfill
\begin{minipage}[b]{0.32\linewidth}
  \centering
  \includegraphics[width=5cm]{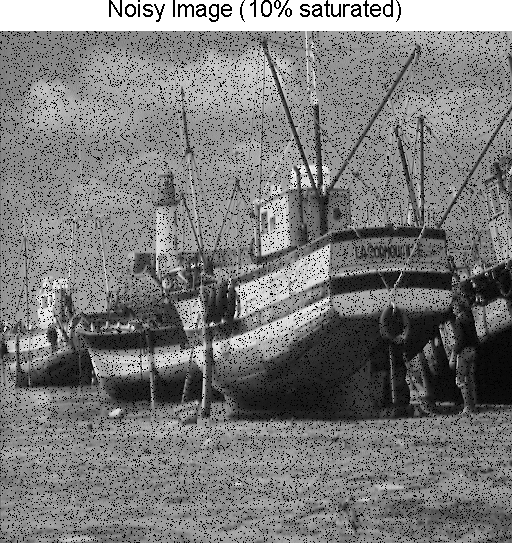}
  \centerline{(b)}\medskip
\end{minipage}
\hfill
\begin{minipage}[b]{0.32\linewidth}
  \centering
  \includegraphics[width=5cm]{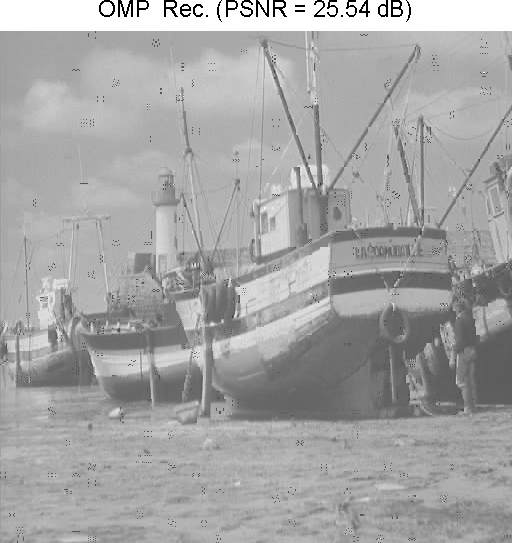}
  \centerline{(c)}\medskip
\end{minipage}
\hfill
\begin{minipage}[b]{0.32\linewidth}
  \centering
  \includegraphics[width=5cm]{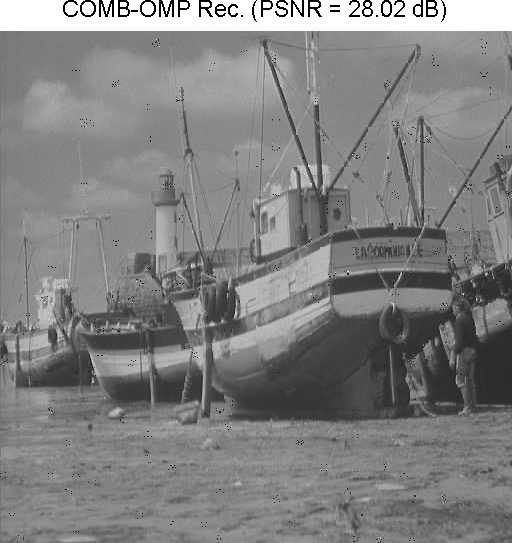}
  \centerline{(d)}\medskip
\end{minipage}
\hfill
\begin{minipage}[b]{0.32\linewidth}
  \centering
  \includegraphics[width=5cm]{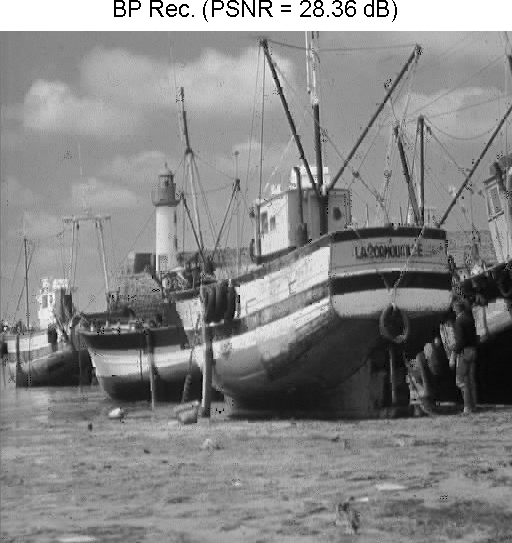}
  \centerline{(e)}\medskip
\end{minipage}
\hfill
\begin{minipage}[b]{0.32\linewidth}
  \centering
  \includegraphics[width=5cm]{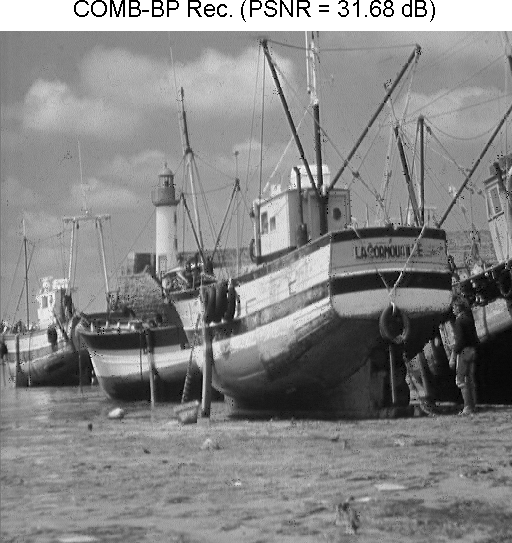}
  \centerline{(f)}\medskip
\end{minipage}

\caption{Actual image, corrupted image ($10\%$ saturation noise) and the recovered images obtained with OMP, COMB-OMP, BP and COMB-BP algorithms. The recovery PSNR for each method are also provided.}
\label{Fig:image_recovery_app}
\end{figure*}

\section{Computational Complexity}
\label{sec:comp_complexity}
In order to quantify the computational complexity, it is essential to consider the implementation specifics. The OMP and COMB-OMP algorithms were directly implemented in MATLAB, whereas the LARS-LASSO \cite{efron2004least} solver in the open-source SPAMS package \cite{mairal2010online} was used to obtain the BP solution. The COMB-BP algorithm and its MATLAB mex interface were implemented by modifying the LARS-LASSO solver in SPAMS. We will focus only on the order complexity of the algorithms since the actual computer time could change based on various factors, such as parallelization etc.

We will first consider the order complexity of the greedy algorithms. Assume a signal $\mathbf{y} \in \mathbb{R}^{M}$ having a $S_g-$sparse representation on the dictionary $\mathbf{G} \in \mathbb{R}^{M \times K_g}$. As described in \cite{mailhe2009low}, at any step $t$, the dominant complexity operations for OMP are: (a) computing the correlation vector, which is of order $M K_g$, (b) computing the Gram matrix for the chosen $t$ atoms, which is of order $tM$, (c) computing the coefficients which is of order $t^2$, and (d) computing the residual, which is of order $tM$. For $S_g$ steps, the total complexity of each of these are of order $S_g M K_g$, $\sum_{t=1}^{S_g} t M$, $\sum_{t=1}^{S_g} t^2$ and $\sum_{t=1}^{S_g} t M$ respectively. Assuming that $K_g > M$ and knowing that $S_g \leq M$, using simple algebra, we obtain the order complexity of OMP to be $S_g M K_g$.

In order to obtain the computational complexity expression for COMB-OMP, we will first compare the steps in OMP and the COMB-OMP algorithms. From Table \ref{tab:comb_omp}, it is clear that the differences between the two algorithms are: (a) in COMB-OMP the coefficient support is updated by performing \textit{max} operations on the two subsets of the $\pi$ array, whereas OMP consists of \textit{max} operation on the whole array, (b) the solution update is constrained to be partially non-negative in COMB-OMP and, (c) a final debiasing step needs to be performed in COMB-OMP which is not the case with OMP. In our implementation, we observed that removing the coefficient sign constraint in the solution update step, and removing the debiasing step do not significantly affect the recovery performance, and the most important consideration is to update the support correctly. Therefore, the complexity of our COMB-OMP implementation is also of order $S_g M K_g$.

For the BP and COMB-BP algorithms, since we use LARS-based solvers, the complexity is fairly straightforward to quantify. LARS-type algorithms update the coefficient support one step at a time using correlations, and perform inversion of Gram matrices similar to OMP. Following through the steps of the LARS algorithm described in \cite[Section 2]{efron2004least}, we can easily identify that the dominant complexity step is to compute the correlations between the dictionary and the residual, similar to OMP. Hence the dominant complexity of this algorithm is also of order $S_g M K_g$. We also note that, since LARS solvers with LASSO modification (LARS-LASSO) can eliminate some coefficients from the active set of coefficients, the actual complexity may be slightly higher. However, this increase cannot be quantified accurately.

\begin{table}[t]
  \centering
  \caption{Recovery performance of the algorithms in the presence of saturation noise whose sign is known. For each image and percentage of saturation noise level, the PSNR (dB) obtained with the OMP, COMB-OMP, COMB-BP and BP algorithms are given in clockwise order beginning from top left corner.}
\begin{tabular}{|c||c|c||c|c||c|c||c|c||c|c|}
\hline 
\textbf{\% Sat} & \multicolumn{2}{c||}{\textbf{Barbara}} & \multicolumn{2}{c||}{\textbf{Boat}} &\multicolumn{2}{c||}{\textbf{House}}&\multicolumn{2}{c||}{\textbf{Lena}}& \multicolumn{2}{c|}{\textbf{Peppers}}  \\ 
\hline 
\hline
\multirow{2}[0]{*}{5} & 30.41 & 32.39 & 30.29 & 32.06 & 32.02 & 35.65 & 32.79 & 35.11 & 28.43 & 30.02 \\
\cline{2-11}
 & 30.89 & \textbf{35.38} & 30.95 & \textbf{35.58} & 32.61 & \textbf{38.98} & 33.46 & \textbf{38.55} & 28.88 & \textbf{33.33} \\
 \hline 
 \hline
 \multirow{2}[0]{*}{10} & 25.48 & 27.73 & 25.44 & 27.77 & 26.78 & 29.18 & 26.93 & 29.95 & 25.11 & 27.63 \\
\cline{2-11} 
 & 27.83 & \textbf{30.92} & 28.45 & \textbf{32.00} & 30.35 & \textbf{34.90} & 30.97 & \textbf{34.84} & 27.35 & \textbf{31.04} \\
 \hline 
 \hline
\multirow{2}[0]{*}{15} & 19.64 & 21.93 & 19.54 & 21.75 & 19.59 & 22.10 & 20.44 & 22.91 & 19.86 & 21.83 \\
\cline{2-11}
& 24.19 & \textbf{26.53} & 25.04 & \textbf{27.44} & 26.28 & \textbf{29.14} & 27.21 & \textbf{29.88} & 23.76 & \textbf{26.66} \\
\hline 
\hline
\multirow{2}[0]{*}{20} & 15.72 & 17.32 & 15.24 & 16.94 & 14.92 & 16.64 & 15.92 & 17.74 & 15.47 & 17.22 \\
\cline{2-11}
 & 19.78 & \textbf{21.72} & 19.96 & \textbf{22.01} & 20.00 & \textbf{22.28} & 21.36 & \textbf{23.68} & 19.51 & \textbf{21.65} \\
 \hline 
 \hline
\multirow{2}[0]{*}{25} & 13.13 & 14.39 & 12.60 & 13.93 & 12.34 & 13.74 & 13.06 & 14.38 & 12.89 & 14.13 \\
\cline{2-11}
& 15.80 & \textbf{17.26} & 15.78 & \textbf{17.30} & 15.91 & \textbf{17.53} & 16.56 & \textbf{18.34} & 15.68 & \textbf{17.25} \\
\hline 
\end{tabular} 
  \label{tab:recovery_perf_images}%
\end{table}%

\section{Experiments}
\label{sec:exp}
The COMB-BP and the COMB-OMP algorithms incorporate the prior knowledge that a set of coefficients to be recovered are non-negative. If we use BP and OMP algorithms for recovery, this prior knowledge cannot be exploited. In order to establish that this additional information leads to improvement in recovery performance, we performed numerical experiments by realizing the elements of dictionary $\mathbf{G} \in \mathbb{R}^{M \times K_g}$ from an i.i.d. zero-mean, unit-variance, Gaussian distribution. The non-zero elements in the coefficient vector $\boldsymbol{\delta}$ were realized from a random sign distribution ($\pm 1$) or from a uniform distribution. We varied the proportion of the non-negative and the unconstrained coefficients in the combined representation and tested the performance of OMP, COMB-OMP, BP, and COMB-BP algorithms in exact and approximate recovery. When measuring the performance of the algorithms in approximate recovery, we also considered various additive noise conditions. We rigorously quantified the performance of the algorithms in compressed recovery using empirical phase transition diagrams \cite{donoho2010precise}. These diagrams provide an accurate picture of the performance of compressed recovery algorithms for various levels of sparsity and number of measurements. Finally, we demonstrate the utility of the proposed algorithms using a real-world application where the goal is to recover an image corrupted with sparse saturation noise whose sign is known. 

\subsection{Exact Coefficient Recovery}
\label{sec:exact_rec}
For these experiments, we fixed the total number of atoms in $\mathbf{G}$ at $K_g = 200$ and the dimension of the signal $M$ at $100$. The three cases tested were (a) $K_x = 50$, $K_d = 150$, (b) $K_x = 100$, $K_d = 100$, and (c) $K_x = 150$, $K_d = 50$. For each trial, the locations of the non-zero coefficients in $\boldsymbol{\delta}$ were chosen uniformly at random. The non-zero coefficients themselves were realized either from a uniform distribution or a sign distribution. For the uniform distribution case, the non-negative coefficients were obtained from the uniform distribution $U(0,1)$ and the general coefficients were obtained from $U(-1,1)$. For the sign distribution case, the general coefficients were obtained with equiprobable positive and negative signs, and the non-negative coefficients were fixed as $1$. The number of non-zero coefficients $S_x$ and $S_d$ were varied from $1$ to $30$ each, and hence the total number of non-zero coefficients, $S_g$, varied from $2$ to $60$. For each $\{S_x,S_d\}$ pair, $1000$ trials were performed. In each trial, $\mathbf{y}$ was obtained using the combined model (\ref{eqn:nonneg_rep_sparse_corr}), and coefficient recovery was performed using the four algorithms. 

While recovering each coefficient vector, the OMP and COMB-OMP were constrained to stop when either the maximum number of non-zero coefficients was $M$, or when the $\ell_2$ norm of the residual vector, $\|\mathbf{y}-\mathbf{G}\hat{\boldsymbol{\delta}}\|_2 < 10^{-6}$. For the BP and COMB-BP algorithms, only the residual error energy constraint was used. Note that in either case, we did not explicitly specify the number of general coefficients and non-zero coefficients. In order to measure the performance of coefficient recovery, we define the relative recovery error between the recovered coefficient vector $\hat{\boldsymbol{\delta}}$ and the actual coefficient vector ${\boldsymbol{\delta}}$ as
\begin{equation}
\label{eqn:rel_rec_err}
RRE = \frac{\|\hat{\boldsymbol{\delta}} - \boldsymbol{\delta}\|_2^2}{\|\boldsymbol{\delta}\|_2^2}.
\end{equation}  If the $RRE$ is less than $10^{-6}$, the coefficient is said to be recovered exactly.

Let us first consider the exact recovery case when the coefficients were realized from the uniform distribution. From Figure \ref{Fig:gauss_unif_exact_rec}, it can be seen that as $K_x$ increases, the performance of COMB-BP and COMB-OMP become increasingly better when compared to that of BP and OMP respectively. In particular, the performance of COMB-BP substantially improves as the non-negative component in the representation becomes bigger. The experiments clearly show that COMB-BP and COMB-OMP perform better than BP and OMP respectively. Furthermore, the presence of a large non-negative component substantially improves the recovery performances of COMB-BP and COMB-OMP. For the OMP and BP algorithms, an average deterministic sparsity threshold of $1$ was obtained for all combinations of $\{K_x,K_g\}$, using $1000$ realizations of $\mathbf{G}$. This threshold was computed using Corollary 4 in \cite{Kuppinger2011}, which applies to any pair of general dictionaries. From Section \ref{sec:comb_rec_l1_cond}, we know that the threshold for COMB-BP is the same as that of OMP and BP. For COMB-OMP also, the average sparsity threshold is $1$, when computed using (\ref{eqn:coh_cond2_comb_omp}). The empirical performance of the four algorithms is similar when the non-zero coefficients were realized from a random sign distribution, as observed from Figure \ref{Fig:gauss_signs_exact_rec}.

\subsection{Approximate Coefficient Recovery}
\label{sec:approx_rec}
The experimental setting for approximate coefficient recovery was similar to that in Section \ref{sec:exact_rec} and for each $\{S_x,S_d\}$ pair, a total of $1000$ trials were performed. However, we considered only the uniform distribution for the non-zero coefficients and the performance was measured by averaging the RRE over all the trials. It can be observed from Figure \ref{Fig:gauss_unif_rec_err} that the COMB-BP and COMB-OMP algorithms exhibit lesser average RRE when compared to their unconstrained counterparts. In this case, the gap in performance between OMP and COMB-OMP is very pronounced when $K_x$ is large.

We also considered the approximate recovery performance in noisy data conditions when $K_x=100$ and $K_d=100$. The data was corrupted with additive Gaussian noise such that the signal-to-noise ratios (SNRs) were $\{0,5,15,25\}$ dBs. The approximate recovery performance is given in Figures \ref{Fig:gauss_unif_rec_err_noisy} (a)-(d). Under high noise conditions, the greedy algorithms perform better than the convex algorithms for a small number of nonzero coefficients, but as the number of non-zero coefficients increase, the performance of $\ell_1$-based algorithms become better. Furthermore, increasing the additive noise reduces the coefficient recovery performance, expectedly. Note that, for all noise conditions, incorporating knowledge of coefficient signs always leads to improvement in recovery performance.
 
\subsection{Phase Transition Characteristics}
\label{sec:phase_tran}
Since one of the main applications of the proposed algorithms is in compressed recovery, we obtained the empirical phase transition characteristics of the algorithms to quantify the effect of number of measurements and sparsity in the recovery performance. The measurement matrix $\mathbf{G} \in \mathbb{R}^{M \times K_g}$ was realized from a Gaussian distribution with $K_g = 400$. It is assumed to be a concatenation of the non-negative and general dictionaries with sizes $K_x=200$ and $K_d=200$ respectively. The non-zero coefficients were obtained from uniform distributions. The number of measurements $M$ were varied between $20$ and $400$ in increments of $20$. For each measurement, we fixed the number of non-zero coefficients $K_g$ as $\rho M$, where $\rho$ is the sparsity factor varied between $0.05$ and $1$ in increments of $0.05$. $200$ trials were conducted for each combination of $M$ and $K_g$. In each trial, $S_x$ was chosen as $\lfloor \tau S_g \rfloor$, where $\tau$ is chosen uniformly randomly between $0$ and $1$, $\lfloor. \rfloor$ is the \textit{rounding-down} function, and $S_d = S_g-S_x$. The number of exact recoveries were measured for each combination of $M$ and $S_g$. For each $M$, the phase transition between the recovery and the no-recovery regime happens at that $S_g$ for which there is a $50 \%$ chance of recovery. This is given in Figure \ref{Fig:phase_tr_char} for the four algorithms, where the contours for recovery probabilities of $0.25$, $0.5$, and $0.75$ are provided. Careful analysis of the diagrams show that incorporating partial knowledge about coefficient signs improves the phase transition characteristics for both greedy and $\ell_1$-based algorithms, and it is more pronounced for convex algorithms.

\subsection{Application in Image Recovery}
\label{sec:image_patch_rec}
We consider a demonstrative application to illustrate possible improvements that can be achieved in image recovery by incorporating partial knowledge of coefficient signs. Note that we are not comparing this with the state-of-the-art but rather only aim to demonstrate the possible utility of our algorithms. The goal here is to recover an image corrupted with saturation noise of known sign. This could involve setting a small fraction of pixels to either $0$ or $255$, which are the minimum and maximum values for the pixels in standard grayscale images. In our case, we saturate 5, 10, 15, 20, and 25$\%$ of the pixels to $255$. 

To recover the image, we divided it into non-overlapping patches of size $8 \times 8$ and vectorized them. The image patches themselves are assumed to be sparse in an orthonormal DCT basis of $64$ atoms and the saturation noise is assumed to be non-negative and sparse in the basis $\mathbf{X} = -\mathbf{I}$, where $\mathbf{I} \in \mathbb{R}^{64 \times 64}$ is an identity matrix. We recovered the coefficients $\boldsymbol{\alpha}$ and $\boldsymbol{\beta}$ using the four algorithms with the stopping criteria, $\|\mathbf{y}-\mathbf{G}\hat{\boldsymbol{\delta}}\|_2 < 10^{-6}$. Each recovered image patch is given by $\mathbf{D}\boldsymbol{\beta}$, using its coefficient vector $\boldsymbol{\beta}$ for the DCT basis. The recovered patches were then arranged to reconstruct the image. Since $\mathbf{D}$ and $\mathbf{X}$ are orthonormal, we have  $\mu_x=\mu_d = 0$, and $\mu_g$ is computed as $0.2405$. The deterministic sparsity thresholds for OMP, BP and COMB-BP are computed as $3$, using Corollary 4 in \cite{Kuppinger2011}. For COMB-OMP, this threshold is $2$, as obtained using (\ref{eqn:coh_cond2_comb_omp}). However, these thresholds are again quite pessimistic, since even a $5 \%$ noise level leads to more than $3$ non-zero coefficients for $\mathbf{X}$ alone. The empirical performance is much better, as described in the experiments below.

The recovery performance of various algorithms is tabulated in Table \ref{tab:recovery_perf_images} for the standard images \textit{Barbara}, \textit{Boat}, \textit{House}, \textit{Lena} and \textit{Peppers256} obtained from the SIPI database \cite{uscsipi}. For each image and noise level, the PSNR in dB for the OMP, COMB-OMP, COMB-BP and BP algorithms are given in clockwise order starting from the upper left corner. Clearly, incorporating partial knowledge of coefficient signs improves the recovery performance, sometimes in the excess of 4 dB. The original, $10\%$ saturated and recovered \textit{Boat} images are shown in Figure \ref{Fig:image_recovery_app}, which supports the same argument. This is particularly apparent when we compare the images recovered with OMP and COMB-OMP algorithms.

\section{Conclusions}
\label{sec:conc}
We considered the problem of recovering sparse solutions from a overcomplete linear model when the solution vector was constrained to be either completely or partially non-negative. When the solution was completely non-negative we derived conditions, based on the theory of polytopes, on the dictionary for the existence of a unique solution. In the case of combined sparse representations, we considered cases when the coefficient support was completely known, partially known or completely unknown. When the coefficient support was completely unknown, we proposed the COMB-OMP algorithm and derived the deterministic sparsity threshold that guarantees recovery of the unique, minimum $\ell_0$ norm solution. Experimental results, using dictionaries drawn from a Gaussian ensemble and non-zero coefficients realized from a uniform distribution or a equiprobable distribution of signs, showed that the COMB-BP and COMB-OMP algorithms perform better in terms of exact and approximate recovery compared to their unconstrained counterparts. We also characterized the recovery performance rigorously using phase transition diagrams and provided a demonstrative real-world application in image recovery for the proposed algorithms. A possible direction for future work is to derive probabilistic sparsity thresholds for the COMB-BP algorithm, under appropriate assumptions on the dictionary and the coefficient vectors, that will explain the improved experimental performance of sparse recovery algorithms with non-negativity constraints when compared to their unconstrained versions.

%
%
%

\appendix
\section{Proof of Lemma \ref{lem:comb_omp_l1_cond}}
\label{app:comb_omp_l1_cond}
For COMB-OMP to recover the unique sparsest representation, no atom from $\mathbf{G}_2$ must enter the support set $\mathcal{G}_t$ at any iteration. Therefore, the residual $\mathbf{r}_t$ at each iteration $t$ must satisfy the condition
\begin{equation}
\label{eqn:comb_rep_omp_cond}
\rho(\mathbf{r}_t) \equiv \frac{\max(\max(\mathbf{X}_2^T\mathbf{r}_t, \mathbf{0}), \|\mathbf{D}_2^T \mathbf{r}_t\|_\infty)}
                          {\max(\max(\mathbf{X}_1^T\mathbf{r}_t, \mathbf{0}), \|\mathbf{D}_1^T \mathbf{r}_t\|_\infty)} < 1.
\end{equation} Since
\begin{equation}
\nonumber
\max(\max(\mathbf{X}_2^T\mathbf{r}_t, \mathbf{0}), \|\mathbf{D}_2^T \mathbf{r}_t\|_\infty) \leq \|\mathbf{G}_2^T \mathbf{r}_t\|_{\infty},
\end{equation}  $\rho(\mathbf{r}_t)$ can be bounded as
\begin{align}
\label{eqn:comb_rep_omp_cond2}
\rho(\mathbf{r}_t) &\leq \frac{\|\mathbf{G}_2^T \mathbf{r}_t\|_{\infty}}
                          {\max(\max(\mathbf{X}_1^T\mathbf{r}_t, \mathbf{0}), \|\mathbf{D}_1^T \mathbf{r}_t\|_\infty)}\\
                          \label{eqn:comb_rep_omp_cond3}
&= \frac{\|\mathbf{G}_2^T (\mathbf{G}_1^{\dagger})^T \mathbf{G}_1^T \mathbf{r}_t\|_{\infty}}
                          {\max(\max(\mathbf{X}_1^T\mathbf{r}_t, \mathbf{0}), \|\mathbf{D}_1^T \mathbf{r}_t\|_\infty)}\\
                          \label{eqn:comb_rep_omp_cond4}
&\leq \frac{\|\mathbf{G}_2^T (\mathbf{G}_1^{\dagger})^T\|_{\infty,\infty} \|\mathbf{G}_1^T \mathbf{r}_t\|_{\infty}}
                          {\max(\max(\mathbf{X}_1^T\mathbf{r}_t, \mathbf{0}), \|\mathbf{D}_1^T \mathbf{r}_t\|_\infty)}\\
                          \label{eqn:comb_rep_omp_cond5}
&= \|\mathbf{G}_2^T (\mathbf{G}_1^{\dagger})^T\|_{\infty,\infty} \\
\label{eqn:comb_rep_omp_cond6}
&= \|\mathbf{G}_1^{\dagger} \mathbf{G}_2 \|_{1,1}    \\
\label{eqn:comb_rep_omp_cond7}
&= \max_{i \in \mathcal{G}^c} \|\mathbf{G}_1^{\dagger} \mathbf{g}_i \|_{1}      
\end{align} Eqn. (\ref{eqn:comb_rep_omp_cond3}) holds since $(\mathbf{G}_1^{\dagger})^T \mathbf{G}_1^T$ is an orthoprojector onto the column space of $\mathbf{G}_1$. Both $\mathbf{y}$ and $\mathbf{G}\boldsymbol{\delta_t}$ lie in the column space  of $\mathbf{G}_1$ and hence  $\mathbf{r}_t$ lies in the same space. The properties of $\|.\|_{\infty,\infty}$ ensures that (\ref{eqn:comb_rep_omp_cond4}) is true. By assumption, the denominator of (\ref{eqn:comb_rep_omp_cond4}) equals $\|\mathbf{G}_1^T\mathbf{r}_t\|_{\infty} $.  Therefore, (\ref{eqn:comb_rep_omp_cond5}) holds true and (\ref{eqn:comb_rep_omp_cond6}) follows from relation $\|\mathbf{A}^T\|_{\infty,\infty} = \|\mathbf{A}\|_{1,1}$ for any matrix $\mathbf{A}$. From (\ref{eqn:comb_rep_omp_cond}), (\ref{eqn:comb_rep_omp_cond2}) and (\ref{eqn:comb_rep_omp_cond7}), the sufficient condition provided in (\ref{eqn:comb_omp_l1_cond}) is obtained.

\section{Proof of Lemma \ref{lem:coh_cond1_comb_omp}}
\label{app:coh_cond1_comb_omp}
For (\ref{eqn:comb_omp_res_rel}) to be satisfied, the sufficient condition is that
\begin{equation}
\label{eqn:comb_omp_nn_res_corr}
\max(\mathbf{X}_1^T\mathbf{r}_t, \mathbf{0}) = \|\mathbf{X}_1^T \mathbf{r}_t\|_\infty.
\end{equation} Therefore, we only have to consider the case where an atom from $\mathbf{X}_1$ will be picked. Let us denote $\boldsymbol{\delta}_1 = [\boldsymbol{\alpha}_1^T \quad \boldsymbol{\beta}_1^T]^T$ and $\mathbf{z} = \mathbf{X}_1^T \mathbf{y} = \mathbf{X}_1^T \mathbf{G}_1 \boldsymbol{\delta}_1$. We will derive the bounds on the maximum positive value, $z_m$, and the minimum negative value, $z_n$, of $\mathbf{z}$.  We denote the smallest possible lower bound on $z_m$ as $\hat{z}_m$, and the largest possible lower bound on $|z_n|$ as $\hat{z}_n$. The worst-case guarantee for  (\ref{eqn:comb_omp_nn_res_corr}) to be true is
\begin{align}
\label{eqn:rel_neg_pos_bound}
\hat{z}_m > \hat{z}_n. 
\end{align} 

Using the fact that
\begin{equation}
\nonumber
\mathbf{X}_1^T\mathbf{G}_1 =  [\mathbf{I}_{S_x} \quad \mathbf{0}] + [\mathbf{X}_1^T\mathbf{X}_1-\mathbf{I}_{S_x} \quad \mathbf{X}_1^T\mathbf{D}_1],
\end{equation} the correlation vector $\mathbf{z}$ can be expressed as
\begin{align}
\label{eqn:corr_vec}
\mathbf{z} &= \boldsymbol{\alpha}_1 + [\mathbf{X}_1^T\mathbf{X}_1-\mathbf{I}_{S_x} \quad \mathbf{X}_1^T\mathbf{D}_1]\boldsymbol{\delta}_1.
\end{align} Let us define $\mathbf{C}_1 = [\mathbf{X}_1^T\mathbf{X}_1-\mathbf{I}_{S_x} \quad \mathbf{X}_1^T\mathbf{D}_1]$ and the elementwise bounds on the submatrices are,
\begin{align}
\nonumber |\mathbf{X}_1^T\mathbf{X}_1-\mathbf{I}_{S_x}| &\leq \mu_x(\mathbf{1}_{S_x,S_x}-\mathbf{I}_{S_x}) \\
\nonumber &\leq \mu_d(\mathbf{1}_{S_x,S_x}-\mathbf{I}_{S_x}),
\end{align} and $\mathbf{X}_1^T\mathbf{D}_1 \leq |\mu_g \mathbf{1}_{S_x,S_d}|$. It is clear that the maximum row sum of $\mathbf{C}_1$ is
\begin{equation}
\label{eqn:maxrowsum_C1}
\|\mathbf{C}_1\|_{\infty,\infty} \leq (S_x - 1)\mu_d + S_d \mu_g.
\end{equation}

In order to derive the smallest lower bound on $z_m$, we will assume that all the coefficients in $\boldsymbol{\delta}_1$ have the same absolute value given by $\alpha$, and hence, from (\ref{eqn:corr_vec}), we have
\begin{align} 
\nonumber
z_m &\geq \alpha - \alpha \|\mathbf{C}_1\|_{\infty,\infty}\\
\label{eqn:pos_max_bound1}
&\geq \alpha(1 - [(S_x - 1)\mu_d + S_d \mu_g]) \equiv \hat{z}_m.
\end{align} The required bound on $z_n$ can be obtained by setting one element of $\boldsymbol{\alpha}_1$ as $\hat{\alpha}$, where $ 0 < \hat{\alpha} < \alpha$, and the absolute value of all the other elements of $\boldsymbol{\delta}_1$ as $\alpha$. We now have
\begin{align}
\nonumber
z_n \geq \hat{\alpha} - \alpha\|\mathbf{C}_1\|_{\infty,\infty}, 
\end{align} and as $\hat{\alpha} \rightarrow 0$,
\begin{align}
\label{eqn:neg_min_bound1}
|z_n| < \alpha [(S_x - 1)\mu_d + S_d \mu_g] \equiv \hat{z}_n,
\end{align} which is the largest possible lower bound on $z_n$. Subsituting  (\ref{eqn:pos_max_bound1}) and (\ref{eqn:neg_min_bound1}) in (\ref{eqn:rel_neg_pos_bound}), results in (\ref{eqn:coh_cond1_comb_omp}).

\section{Proof of Lemma \ref{lem:rec_cond_comb_omp}}
\label{app:rec_cond_comb_omp}
The condition for success of COMB-OMP can be written as
\begin{align}
\label{eqn:rec_cond_comb_omp1}
\max_{i \in \mathcal{G}^c} \|\mathbf{G}_1^{\dagger} \mathbf{g}_i \|_{1} 
&\leq \|(\mathbf{G}_1^T\mathbf{G}_1)^{-1}\|_{1,1} \max_{i \in \mathcal{G}^c} \|\mathbf{G}_1^T \mathbf{g}_i\|_1,
\end{align} using the property of $\|.\|_{1,1}$ and the fact that $\mathbf{G}_1^{\dagger} = (\mathbf{G}_1^T\mathbf{G}_1)^{-1}\mathbf{G}_1^T$.

In order to compute the lower bound for $\|(\mathbf{G}_1^T\mathbf{G}_1)^{-1}\|_{1,1}$, we first expand the Gramm matrix
\begin{equation}
\label{eqn:gram_mat_exp} 
 \mathbf{G}_1^T \mathbf{G}_1 = \mathbf{I}_{S_g} + \mathbf{C},
\end{equation} where
\begin{align}
\nonumber
\mathbf{C} \equiv \left[ \begin{array}{cc}
\mathbf{X}_1^T\mathbf{X}_1-\mathbf{I}_{S_x} & \mathbf{X}_1^T\mathbf{D}_1 \\
\mathbf{D}_1^T\mathbf{X}_1 & \mathbf{D}_1^T\mathbf{D}_1-\mathbf{I}_{S_d}
\end{array} \right].
\end{align} $\mathbf{C}$ can be bounded elementwise as
\begin{align}
\nonumber
|\mathbf{C}| &\leq \left[ \begin{array}{cc}
\mu_x(\mathbf{1}_{S_x,S_x}-\mathbf{I}_{S_x}) & \mu_g \mathbf{1}_{S_x,S_d} \\
\mu_g \mathbf{1}_{S_d,S_x} & \mu_d(\mathbf{1}_{S_d,S_d}-\mathbf{I}_{S_d})
\end{array} \right] \\
\nonumber
&\leq \left[ \begin{array}{cc}
\mu_d(\mathbf{1}_{S_x,S_x}-\mathbf{I}_{S_x}) & \mu_g \mathbf{1}_{S_x,S_d} \\
\mu_g \mathbf{1}_{S_d,S_x} & \mu_d(\mathbf{1}_{S_d,S_d}-\mathbf{I}_{S_d})
\end{array} \right],
\end{align} since $\mu_x \leq \mu_d$ by assumption. The maximum column sum of $\mathbf{C}$ is bounded as 
\begin{equation}
\label{eqn:maxrowsum_C}
\|\mathbf{C}\|_{1,1} \leq (S_x - 1)\mu_d + S_d \mu_g,
\end{equation} since $S_x \leq S_d$. From (\ref{eqn:gram_mat_exp}), we also observe that $\|\mathbf{C}\|_{1,1} < 1$, since $\mathbf{G}_1^T \mathbf{G}_1$ is strictly diagonally dominant, because of the linear independence of the columns of $\mathbf{G}_1$. Using (\ref{eqn:gram_mat_exp}), we can write
\begin{align}
\label{eqn:rec_cond_comb_omp2a}
\|(\mathbf{G}_1^T\mathbf{G}_1)^{-1}\|_{1,1}  = \|(\mathbf{I}_{S_g} + \mathbf{C})^{-1}\|_{1,1}.
\end{align} The Neumann series $\sum_k (-\mathbf{C})^k$ converges to $(\mathbf{I}_{S_g} + \mathbf{C})^{-1}$, whenever $\|\mathbf{C}\|_{1,1} < 1$ \cite{Kreyszig1989} and hence (\ref{eqn:rec_cond_comb_omp2a}) can be expressed as
\begin{align}
\nonumber
\|(\mathbf{G}_1^T\mathbf{G}_1)^{-1}\|_{1,1}  &= \left\|\sum_{k=1}^{\infty}(-\mathbf{C})^k \right\|_{1,1} \\
\nonumber
&\leq \sum_{k=1}^{\infty}\left\|\mathbf{C}\right\|_{1,1}^k  \\
\nonumber
&= \frac{1}{1-\|\mathbf{C}\|_{1,1}}\\
\label{eqn:rec_cond_comb_omp5}
&\leq \frac{1}{1-(S_x \mu_d + S_d \mu_g - \mu_d)},
\end{align} using (\ref{eqn:maxrowsum_C}). If $\mathbf{g}_i$ is a vector chosen from $\mathbf{X}_2$, $|\mathbf{G}_1^T \mathbf{g}_i| \leq [\mu_d \mathbf{1}_{S_x}^T \quad \mu_g \mathbf{1}_{S_d}^T]^T$ and hence we have
\begin{align}
\label{eqn:rec_cond_comb_omp6}
\max_{i \in \mathcal{G}^c} \|\mathbf{G}_1^T \mathbf{g}_i\|_1 \leq S_x \mu_d + S_d \mu_g = (S_x+S_d) \mu_d + S_d (\mu_g-\mu_d).
\end{align} If $\mathbf{g}_i$ is chosen from $\mathbf{D}_2$, $|\mathbf{G}_1^T \mathbf{g}_i| \leq [\mu_g \mathbf{1}_{S_x}^T \quad \mu_d \mathbf{1}_{S_d}^T]^T$. Therefore
\begin{align}
\label{eqn:rec_cond_comb_omp7}
\max_{i \in \mathcal{G}^c} \|\mathbf{G}_1^T \mathbf{g}_i\|_1 \leq S_x \mu_g + S_d \mu_d = (S_x+S_d) \mu_d + S_x (\mu_g-\mu_d).
\end{align} Since $S_d \geq S_x$, among (\ref{eqn:rec_cond_comb_omp6}) and (\ref{eqn:rec_cond_comb_omp7}), we will choose (\ref{eqn:rec_cond_comb_omp6}) as our bound. Substituting (\ref{eqn:rec_cond_comb_omp6}) and (\ref{eqn:rec_cond_comb_omp5}) in (\ref{eqn:rec_cond_comb_omp1}), we can obtain (\ref{eqn:rec_cond_comb_omp}) as the condition for COMB-OMP to succeed.

\section{Proof of Lemma \ref{lem:coh_cond_comb_omp_overall}}
\label{app:coh_cond_comb_omp_overall}

Rewriting (\ref{eqn:rec_cond_comb_omp}), we obtain
\begin{equation}
\nonumber
S_x < \frac{1+ \mu_d - 2 S_d \mu_g}{2 \mu_d} \equiv \psi(S_d).
\end{equation} The threshold on the total number of non-zero coefficients, $S_g$, can be obtained by minimizing $\psi(S_d)+S_d$ over $S_d$. The constraint is that $S_d \geq 1$ since $S_x \leq S_d$ and the overall representation will have at least one non-zero coefficient. Denoting $S$ to be the sparsity threshold, it can be obtained as
\begin{align}
\nonumber
S = \min_{S_d \geq 1} [S_d+\psi(S_d)].
\end{align} Relaxing the constraint that $S_d$ is an integer, the minimum will be obtained when $\mu_d = \mu_g = \mu_m$, where $\mu_m \equiv \max(\mu_d,\mu_g)$ and the minimum value is
\begin{align}
\nonumber
S = 0.5 \left(1+\frac{1}{\mu_m}\right),
\end{align} which is the strict upper bound on $S_g$. Since $\mu_d \geq \mu_x$ by assumption, we can generalize the definition of $\mu_m$ as $\mu_m \equiv \max(\mu_d,\mu_g, \mu_x)$. Note that this worst-case sparsity threshold automatically makes the matrix of chosen atoms $\mathbf{G}_1$ to be full-rank, thereby satisfying the requirement posed by Lemma \ref{lem:rec_cond_comb_omp}. This is because, in this case the Gram matrix $\mathbf{G}_1^T \mathbf{G}_1$ will be strictly diagonally dominant and hence non-singular.

\section{Proof of Lemma \ref{lem:coh_cond2_comb_omp}}
\label{app:coh_cond2_comb_omp}

Since we defined $\mu_m \equiv \max(\mu_g,\mu_d,\mu_x)$, we will assume that the overall coherence of $\mathbf{G}_1$ is $\mu_m$ and the total number of columns in $\mathbf{G}_1$ is $S_g$. We will denote $\mathbf{G}_1 = [\mathbf{G}_a \quad \mathbf{G}_b]$, where $\mathbf{G}_a$ with $S_a$ columns contains the atoms already chosen for the representation and $\mathbf{G}_b$ contains $S_b = S_g-S_a$ atoms, one of which will be chosen by the residual. The residual $\mathbf{r}_t$ will be simply denoted as $\mathbf{r}$ for notational convenience and is obtained using a least squares procedure as 
\begin{equation}
\label{eqn:residual_ls}
\mathbf{r} = \mathbf{y}-\mathbf{G}_a \boldsymbol{\delta}_a,
\end{equation} where
\begin{equation}
\label{eqn:coeff_ls} 
\boldsymbol{\delta}_a = \mathbf{G}_a^{\dagger} \mathbf{y}.
\end{equation}

Let us denote the correlation vector as
\begin{equation}
\label{eqn:corr_vec_res}
\mathbf{z} = \mathbf{G}_b^T \mathbf{r}.
\end{equation} Similar to the proof of Lemma \ref{lem:coh_cond1_comb_omp}, we are only concerned about recovering the non-negative coefficients as it will give us sufficient conditions under which (\ref{eqn:comb_omp_res_rel}) will be satisfied. The bounds on the maximum positive value, $z_m$, and the minimum negative value, $z_n$, of $\mathbf{z}$ will be derived assuming that all the elements in $\boldsymbol{\delta}_b$ are constrained to be non-negative. In this case, the smallest possible lower bound on $z_m$, given by $\hat{z}_m$ and the largest possible lower bound on $|z_n|$ given by $\hat{z}_n$.  For (\ref{eqn:comb_omp_res_rel}) to hold for any $\mathbf{r}_t$, where $t\geq 1$, similar to the proof of Lemma \ref{lem:coh_cond1_comb_omp}, we need to show that 
\begin{equation}
\label{eqn:zm_gt_zn}
\hat{z}_m > \hat{z}_n.
\end{equation}

We will first expand (\ref{eqn:corr_vec_res}) using (\ref{eqn:residual_ls}) and (\ref{eqn:coeff_ls})
\begin{align}
\nonumber
\mathbf{z} &= \mathbf{G}_b^T(\mathbf{y} - \mathbf{G}_a \mathbf{G}_a^{\dagger} \mathbf{y}) \\
\label{eqn:corr_vec1}
 &= \mathbf{G}_b^T(\mathbf{I} - \mathbf{G}_a \mathbf{G}_a^{\dagger}) \mathbf{G}_1 \boldsymbol{\delta}_1 \\
 \label{eqn:corr_vec2}
 &= \mathbf{G}_b^T(\mathbf{I} - \mathbf{G}_a \mathbf{G}_a^{\dagger}) \mathbf{G}_b \boldsymbol{\delta}_b,
\end{align} which is obtained by substituting $\mathbf{G}_1 = [\mathbf{G}_a \quad \mathbf{G}_b]$ and $\boldsymbol{\delta}_1 = [\boldsymbol{\delta}_a^T \quad \boldsymbol{\delta}_b^T]^T$ in (\ref{eqn:corr_vec1}). Let us denote any two distinct columns from $\mathbf{G}_b$ by $\mathbf{g}_i$ and $\mathbf{g}_j$. Let us also denote the matrix $\mathbf{Q} = \mathbf{G}_b^T(\mathbf{I} - \mathbf{G}_a \mathbf{G}_a^{\dagger})\mathbf{G}_b$, which is of size $S_b \times S_b$, and designate its $(i,j)^{\text{th}}$ element to be $q_{ij}$. The correlation vector in (\ref{eqn:corr_vec2}) can be now expressed as
\begin{align}
\nonumber
\mathbf{z} = \mathbf{Q} \boldsymbol{\delta}_b
\end{align}  We will compute bounds on the diagonal and off-diagonal elements of $\mathbf{Q}$. Expanding $\mathbf{G}^{\dagger}$, we can write
\begin{align}
\label{eqn:corr_gij}
|q_{ij}| &= \left|\mathbf{g}_i^T[\mathbf{I} - \mathbf{G}_a (\mathbf{G}_a^T\mathbf{G}_a)^{-1}\mathbf{G}_a^T] \mathbf{g}_j\right| \\
\label{eqn:corr1_gij}
&\leq \left|\mathbf{g}_i^T \mathbf{g}_j\right| + \left| \mathbf{g}_i^T \mathbf{G}_a (\mathbf{G}_a^T\mathbf{G}_a)^{-1}\mathbf{G}_a^T \mathbf{g}_j\right| \\
\label{eqn:corr2_gij}
& \leq \left|\mathbf{g}_i^T \mathbf{g}_j\right| + \left\| \mathbf{G}_a^T \mathbf{g}_i\right\|_2^2 \left\|(\mathbf{G}_a^T\mathbf{G}_a)^{-1}\right\|_2.
\end{align} Eqn. (\ref{eqn:corr1_gij}) follows from applying triangle inequality on (\ref{eqn:corr_gij}) and  (\ref{eqn:corr2_gij}) is obtained by upper bounding the second term in the right hand side of (\ref{eqn:corr1_gij}). We can express
\begin{align}
\nonumber
\left\| \mathbf{G}_a^T \mathbf{g}_i\right\|_2^2 \leq S_a \mu_m^2
\end{align} since the maximum absolute coherence between any two elements in $\mathbf{G}_1$  is $\mu_m$. Since $\left\|(\mathbf{G}_a^T\mathbf{G}_a)^{-1}\right\|_2 \leq 1/\lambda_{min}(\mathbf{G}_a^T\mathbf{G}_a)$, and by Gershgorin's disc theorem \cite[Theorem 6.1.1]{Horn1985}, $\lambda_{min}(\mathbf{G}_a^T\mathbf{G}_a)  \leq [1-\mu_m(S_a-1)]^{+}$, we can rewrite (\ref{eqn:corr2_gij}) as
\begin{align}
\label{eqn:corr4_gij}
|q_{ij}| \leq \mu_m +  \frac{S_a \mu_m^2}{[1-\mu_m(S_a-1)]^{+}}.
\end{align} When $i=j$, we have
\begin{align}
\label{eqn:corr5_gij}
|q_{ii}| &= \left|\mathbf{g}_i^T[\mathbf{I} - \mathbf{G}_a (\mathbf{G}_a^T\mathbf{G}_a)^{-1}\mathbf{G}_a^T] \mathbf{g}_i\right| \\
\label{eqn:corr6_gij}
&\geq \left|\mathbf{g}_i^T \mathbf{g}_j\right| - \left| \mathbf{g}_i^T \mathbf{G}_a (\mathbf{G}_a^T\mathbf{G}_a)^{-1}\mathbf{G}_a^T \mathbf{g}_j\right| \\
\label{eqn:corr7_gij}
& \geq 1-\frac{S_a \mu_m^2}{[1-\mu_m(S_a-1)]^{+}}.
\end{align} Eqn. (\ref{eqn:corr6_gij}) follows from applying reverse triangle inequality on (\ref{eqn:corr5_gij}) and (\ref{eqn:corr7_gij}) is obtained by following steps similar to the derivation of upper bound on $|q_{ij}|$. The bounds given by (\ref{eqn:corr4_gij}) and (\ref{eqn:corr7_gij}) are valid only if $1-\mu_m(S_a-1) > 0$, which can be verified by substituting $\mu_m < 1/(2S_g - 1)$, from (\ref{eqn:coh_cond2_comb_omp}), and $S_a < S_g$. Therefore, (\ref{eqn:corr5_gij}) and (\ref{eqn:corr7_gij}) can be rewritten as $|q_{ij}| \leq q_1$, $|q_{ii}| \geq q_2$, where
\begin{align}
\nonumber
q_1 &= C_1  \mu_m \\
\nonumber
q_2 &= C_1 (1-S_a \mu_m),
\end{align} and $C_1 = (1+\mu_m)/[1-\mu_m(S_a-1)]$. 

We know that the diagonal elements of $\mathbf{Q}$ are lower bounded by $q_1$ and the off-diagonal elements are upper bounded by $q_2$. Since there are $S_b$ rows in $\mathbf{Q}$, the smallest lower bound on $z_m$ is
\begin{align}
\label{eqn:wor_case_zm}
z_m \geq \alpha q_2 - \alpha (S_b-1) q_1 \equiv \hat{z}_m
\end{align} which is obtained when all the elements in $\boldsymbol{\delta}_b$ are set to $\alpha$. The required bound on $z_n$ is obtained by setting one element of $\boldsymbol{\delta}_b$ corresponding to a positive coefficient as $\hat{\alpha}$, where $0<\hat{\alpha}<\alpha$, $\hat{\alpha}$ approaches zero and all the other values in $\boldsymbol{\delta}_b$ are set to $\alpha$. $z_n$ can be now bounded as $z_n \geq \hat{\alpha} q_2 - \alpha q_1 (S_b-1)$. As $\hat{\alpha} \rightarrow 0$, we have
\begin{align}
\label{eqn:best_case_zm}
|z_n| < \alpha (S_b-1) q_2 \equiv \hat{z}_n. 
\end{align} Using (\ref{eqn:zm_gt_zn}), (\ref{eqn:wor_case_zm}) and (\ref{eqn:best_case_zm}), we have
\begin{align}
\nonumber
\mu_m < \frac{1}{2 S_g - S_a - 2},
\end{align}which is always satisfied since we know from (\ref{eqn:coh_cond2_comb_omp}) that $\mu_m < 1/(2 S_g-1)$ and $S_a \geq 1$.





\bibliographystyle{elsarticle-num}
\bibliography{refs}







\end{document}